\documentclass[12pt]{amsart}
\usepackage{amsmath, amssymb, amsthm, mathrsfs}
\usepackage{geometry}
\usepackage{fancyhdr}
\usepackage{enumitem}
\usepackage{hyperref}

\newcommand{\pos}{\mathbb{R}_+}
\newcommand{\ca}{\mathcal{A}}
\newcommand{\cx}{\mathcal{X}}
\newcommand{\ct}{\mathcal{T}}
\newcommand{\cl}{\mathcal{L}}
\newcommand{\cu}{\mathcal{U}}
\newcommand{\cs}{\mathcal{S}}

\newcommand{\cp}{\mathcal{P}}
\newcommand{\cb}{\mathcal{B}}

\newcommand{\cd}{\mathcal{D}}
\newcommand{\cq}{\mathcal{Q}}
\newcommand{\ch}{\mathcal{H}}
\newcommand{\crl}{\mathcal{R}_\lambda}
\newcommand{\dds}{\frac{d}{ds}}

\newcommand{\ttt}{(\ct(t))_{t\geq 0}}
\newcommand{\loc}{_{\text{loc}}}
\newcommand{\doo}{\partial\Omega}
\newcommand{\dx}{\partial X}

\newcommand{\norm}[1]{\Vert #1 \Vert}
\newcommand{\rg}{\operatorname{rg}}

\newcommand{\mo}{\mu_1}
\newcommand{\md}{\mu_2}

\newtheorem{theorem}{Theorem}
\newtheorem{assums}[theorem]{Assumtptions}
\newtheorem{proposition}[theorem]{Proposition}
\newtheorem{lemma}[theorem]{Lemma}
\newtheorem{example}[theorem]{Example}
\newtheorem{examples}[theorem]{Examples}
\newtheorem{corollary}[theorem]{Corollary}
\newtheorem{definition}[theorem]{Definition}
\newtheorem{remark}[theorem]{Remark}

\numberwithin{equation}{section}
\numberwithin{theorem}{section}

\begin{document}

\title[THEORY AND APPLICATIONS OF ONE-SIDED COUPLED OPERATOR MATRICES]{THEORY AND APPLICATIONS OF\\
ONE-SIDED COUPLED OPERATOR MATRICES}

 \author{Marjeta Kramar}
\address{Marjeta Kramar\\
University of Ljubljana\\
Faculty of Civil and Geodetic Engineering\\
Department for Mathematics and Physics\\
Jamova 2\\
1000 Ljubljana, Slovenia}
\email{mkramar@fgg.uni-lj.si}

 \author{Delio Mugnolo}
\address{Delio Mugnolo\\
Arbeitsbereich Funktionalanalysis\\
Mathematisches Institut\\
Universität Tübingen\\
Auf Der Morgenstelle 10\\
D- 72076 Tübingen, Germany}
\email{demu@fa.uni-tuebingen. de}

 \author{Rainer Nagel}
\address{Rainer Nagel\\
Arbeitsbereich Funktionalanalysis\\
Mathematisches Institut\\
Universität Tübingen\\
Auf Der Morgenstelle 10\\
D- 72076 Tübingen, Germany}
\email{rana@fa.uni-tuebingen.de}




\thanks{Parts of this paper are inspired by a series of lectures held by Klaus-J. Engel at the summer school of ``Semigroups of operators'' at Cortona in July 2001. We thank him for many inspiring discussions.
The second author is supported by the Istituto Nazionale di Alta Matematica ``Francesco Severi''.\\[3pt]
\textbf{This article  was originally published in: Conf.\ Sem.\ Mat.\ Univ.\ Bari 283 (2003).}}


\begin{abstract}
The theory of one-sided coupled operator matrices, recently introduced by K.-J. Engel, is an abstract framework for concrete initial value problems and allows complete information on well-posedness, and stability of solutions. These notes are meant as a survey on this rich theory, with a particular stress on applications to initial-boundary value problems with unbounded boundary feedbacks. A diffusion-transport system with dynamical boundary conditions is discussed, and its well-posedness and various other properties are investigated. As a by-product, the well-posedness of a wave equation with dynamical boundary condition is also obtained.
\end{abstract}
\maketitle

\section{INTRODUCTION}
\label{sec:introduction}

Of concern is an abstract Cauchy problem
\[
\begin{cases}
\dot{\cu}(t) = \ca\;{\cu}(t), & t \geq 0, \\
{\cu}(0) = {\cu}_0,
\end{cases}
\tag{ACP}\label{eq:ACP}
\]
on a product space $\cx := X\times Y$ for two Banach spaces $X$ and $Y$. It is well-known that \eqref{eq:ACP} can be studied by means of the theory of one-parameter semigroups of bounded linear operators. In particular, \eqref{eq:ACP} is \textit{well-posed} (which, roughly speaking, means that for all ${\cu}_0\in D(\ca)$ there exists a unique solution of \eqref{eq:ACP}, which depends continuously on the initial data) if and only if $\ca$ generates a strongly continuous semigroup on $\cx$ (see \cite[\S~II.6]{EN00}). Thus, it is our goal to check the generator property of $\ca$ (e.g., by using the Hille–Yosida Theorem), and then to discuss the asymptotic behavior of the solutions by means of spectral theory as, e.g., in \cite[Chapters~IV and V]{EN00}.

If we are given an unbounded operator $\ca$ on a product space $\cx = X\times Y$, one can wonder whether it is possible to represent it as an operator matrix and subsequently to characterize properties of $\ca$ by the matrix entries. Having reduced in this way a ``two-dimensional'' problem on $\cx = X\times Y$ to an equivalent ``one-dimensional'' problem on the factor spaces $X$ and/or $Y$ there is the hope to obtain conditions that in practice are easier to check. R. Nagel started such a semigroup theory for operator matrices in \cite{Na85} and \cite{Na89}. The concept of \textit{one-sided coupled operator matrices} turned out to be of particular interest and was introduced and mainly developed by K.-J.~Engel (see \cite{En97a,En97b,En98a,En98b,En99,En02}). Other authors studying such (or similar) operator matrices are Casarino, Engel, Nagel, and Nickel (see \cite{CENN02}), Favini, Goldstein, Ruiz Goldstein, and Romanelli (see \cite{FGGR02}), and Arendt (see \cite[\S~6]{ABHN01}). This theory applies to many initial value problems (including second order Cauchy problems and boundary value problems), but we explain all abstract results by two special classes of equations that will be referred to in the following: Volterra integro-differential equations and delay differential equations. We conclude this paper by discussing concrete boundary feedback systems.

\section{VOLTERRA AND DELAY DIFFERENTIAL EQUATIONS}
\label{sec:volterra-delay}

We consider the following \textit{Volterra integro-differential equation}
\[
\begin{cases}
\dot{u}(t) = Au(t) + \int_0^t C(t-s)Au(s)\,ds + f(t), & t \geq 0, \\
u(0) = u_0,
\end{cases}
\tag{VIDE}\label{eq:VIDE}
\]
where $X$ is a Banach space, $A$ is an operator on $X$, $u,f,C(\cdot)x\in F(\pos,X)$ for all $x\in D(A)$ (where $F$ denotes some function space such as $F=C_0$, $L^p$, etc.) and $u_0\in X$. In order to obtain an appropriate \eqref{eq:ACP} for this equation, we define a new state space $\cx := X \times F(\pos,X)$ and a new operator
\begin{equation}\label{eq:2.1}
\ca:= \begin{pmatrix}
A & \delta_0 \\
C(\cdot) & \dds
\end{pmatrix},\quad
D(\ca):= D(A) \times D\left(\dds\right)
\end{equation}
on $\cx$, where $\delta_0$ is the point evaluation in $0$ and $\dds$ denotes the generator of the left translation semigroup on $F(\pos,X)$ (for a thorough treatment see \cite[\S~VI.7]{EN00}).

As second prototype we consider the \textit{abstract delay differential equation}
\[
\begin{cases}
\dot{u}(t) = Au(t) + \Phi u_t, & t \geq 0, \\
u(s) = h(s), & s\in[-1,0],
\end{cases}
\tag{ADDE}\label{eq:ADDE}
\]
on a Banach space $X$, where $A$ is an operator on $X$, $u:[-1,+\infty)\rightarrow X$,
\[
u_t:[-1,0]\ni s\mapsto u(t+s)\in X
\]
is the \textit{history function} associated to $u$, $h\in L^p([-1,0],X)$ and $\Phi:W^{1,p}([-1,0],X)\rightarrow X$ a \textit{delay operator}. Again, we reformulate this problem as \eqref{eq:ACP} on a product space $\cx := X\times L^p([-1,0],X)$, for the operator
\begin{equation}\label{eq:2.2}
\ca:=\begin{pmatrix}
A& \Phi \\
0& \dds
\end{pmatrix}, \quad
D(\ca):= \left\{ \begin{pmatrix}
x \\ f
\end{pmatrix} \in D(A)\times W^{1,p}\left([-1,0],X\right) : f(0)=x \right\}.
\end{equation}
For the definition of classical solutions of \eqref{eq:ADDE}, and for more results about this class of differential equations, we refer to \cite[\S~VI.6]{EN00} or \cite{BP02}.

\section{GENERAL FRAMEWORK}
\label{sec:general-framework}

Our aim is to develop a general matrix theory which allows to treat the examples from the previous section (and many other) in a unified and systematic way. To that purpose we will always assume the following.

\begin{assums}\label{assumptions:3.1}
\begin{enumerate}[label=\rm{(A$_{\arabic*}$)}, itemsep=2pt, parsep=0pt]
\item $X$ and $Y$ are Banach spaces.
\item $\tilde{A}:D(\tilde{A})\subseteq X\to X$ and $D:D(D)\subseteq Y\to Y$ are closed linear operators.
\item $B:[D(D)]\rightarrow X$ and $L:[D(\tilde{A})]\rightarrow Y$ are bounded linear operators.
\end{enumerate}
\end{assums}

Here, for a given closed operator $C$ on a Banach space $Z$, $[D(C)]$ denotes the Banach space $\left(D(C),\norm{\cdot}_C\right)$ equipped with the graph norm
\[
\norm{z}_C := \norm{z} + \norm{Cz}, \qquad z\in D(C).
\]
In the following we will denote by $I_X$ and $I_Y$ the identity operators on $X$ and $Y$, and by $\pi_X$ and $\pi_Y$ the projections from $\cx:=X\times Y$ onto $X$ and $Y$, respectively.

\begin{definition}\label{def:osc-matrix}
Under the Assumptions~\ref{assumptions:3.1} we call the operator
\[
\ca:=\begin{pmatrix}
\tilde{A} & B \\
0 & D
\end{pmatrix}
\begin{pmatrix}
I_X & 0 \\
L & I_Y
\end{pmatrix}
\]
with domain
\[
D(\ca):=\left\{\begin{pmatrix}
x \\ y
\end{pmatrix}\in D(\tilde{A})\times Y : Lx+y\in D(D)\right\}
\]
the associated \textit{one-sided coupled} (short, \textit{osc}) \textit{operator matrix} on $\cx$.
\end{definition}

\section{CHARACTERIZATION OF ONE-SIDED COUPLED OPERATOR MATRICES}
\label{sec:characterization}

It is now a rather natural question to ask which operators $\ca$ on a product space $\cx$ are osc. For invertible operators the answer is provided by the following result (cf. \cite{En98b}).

\begin{proposition}\label{prop:4.1}
Let $X$, $Y$ be Banach spaces and $(\ca,D(\ca))$ be a linear operator on $\cx: =X\times Y$. Define the operators $B$ and $D$ by
\[
D(B)=D(D):=\left\{ y\in Y : \begin{pmatrix} 0 \\ y \end{pmatrix}\in D(\ca)\right\},
\]
\[
By:=\pi_X\left[\ca\begin{pmatrix} 0 \\ y \end{pmatrix}\right], \qquad
Dy:=\pi_Y\left[\ca\begin{pmatrix} 0 \\ y \end{pmatrix}\right].
\]
If $\ca$ is invertible and $\ca^{-1}=\begin{pmatrix} U & V \\ W & S \end{pmatrix}$, then the following properties are equivalent.
\begin{enumerate}[label=(\alph*)]
\item $\ca$ is an osc operator matrix and $D$ is invertible.
\item $U$ is injective and $\rg(V)\subseteq \rg(U)$.
\item $D$ is invertible.
\end{enumerate}
\end{proposition}

\begin{proof}
Observe first that $\ca \begin{pmatrix} 0 \\ y \end{pmatrix} = \begin{pmatrix} By \\ Dy \end{pmatrix}$ for $y\in D(B)=D(D)$, hence
\begin{equation}\label{eq:4.1}
\begin{pmatrix} 0 \\ y \end{pmatrix} = \ca^{-1} \begin{pmatrix} By \\ Dy \end{pmatrix} = \begin{pmatrix} U & V \\ W & S \end{pmatrix} \begin{pmatrix} By \\ Dy \end{pmatrix} = \begin{pmatrix} UBy + VDy \\ WBy + SDy \end{pmatrix}.
\end{equation}
From this it follows that
\begin{equation}\label{eq:4.2}
UB+VD=0_{D(D)} \quad \text{and} \quad WB+SD = I_{D(D)}.
\end{equation}

We now prove that (c)$\Rightarrow$(b)$\Rightarrow$(a).

(c)$\Rightarrow$(b). Since $D$ is invertible, we obtain from \eqref{eq:4.2} that $V=-UBD^{-1}$ and therefore $\rg(V)\subseteq \rg(U)$. To show the injectivity of $U$ take $x\in\ker(U)$. Then $\ca^{-1}\begin{pmatrix} x \\ 0 \end{pmatrix}=\begin{pmatrix} 0 \\ Wx \end{pmatrix}\in D(\ca)$. This implies $Wx\in D(D)$ and by
\[
\begin{pmatrix} BWx \\ DWx \end{pmatrix}=\ca\begin{pmatrix} 0 \\ Wx \end{pmatrix}=\begin{pmatrix} x \\ 0 \end{pmatrix},
\]
we obtain $x=BWx$ and $Wx\in\ker(D)=\{0\}$. Hence, $x=0$ proving that $U$ is injective.

(b)$\Rightarrow$(a). For $\tilde A:=U^{-1}$ with domain $D(\tilde A):=\rg(U)$ and $L:=-WU^{-1}$, we define the matrix
\[
\begin{aligned}
D(\tilde\ca) &:=\left\{\begin{pmatrix} x \\ y \end{pmatrix}\in D(\tilde A)\times Y : Lx+y\in D(D)\right\}, \\
\tilde\ca &:=\begin{pmatrix} \tilde A & B \\ 0 & D \end{pmatrix}\begin{pmatrix} I_X & 0 \\ L & I_Y \end{pmatrix}.
\end{aligned}
\]
We first prove that $\ca=\tilde\ca$ by showing that $\ca^{-1}$ is the inverse of $\tilde\ca$. In fact, for $\begin{pmatrix} x \\ y \end{pmatrix}\in D(\tilde\ca)$ we have
\[
\begin{aligned}
\ca^{-1}\tilde\ca\begin{pmatrix} x \\ y \end{pmatrix}
&=\begin{pmatrix} U & V \\ W & S \end{pmatrix}\begin{pmatrix} \tilde A & B \\ 0 & D \end{pmatrix}\begin{pmatrix} x \\ Lx+y \end{pmatrix} \\
&=\begin{pmatrix} U & V \\ W & S \end{pmatrix}\begin{pmatrix} U^{-1}x+B(Lx+y) \\ D(Lx+y) \end{pmatrix} \\
&=\begin{pmatrix} x+UB(Lx+y)+VD(Lx+y) \\ WU^{-1}x+WB(Lx+y)+ SD(Lx+y) \end{pmatrix} \\
&=\begin{pmatrix} x+(UB+VD)(Lx+y) \\ WU^{-1}x+(WB+SD)(-WU^{-1}x+y) \end{pmatrix} \\
&\stackrel{\eqref{eq:4.2}}{=}\begin{pmatrix} x \\ y \end{pmatrix}
\end{aligned}
\]
and hence $\tilde\ca\subseteq \ca$. To show equality it suffices to verify that $D(\ca)\subseteq D(\tilde\ca)$, i.e.,
\begin{equation}\label{eq:4.3}
\ca^{-1}\begin{pmatrix} x \\ y \end{pmatrix}=\begin{pmatrix} Ux+Vy \\ Wx+Sy \end{pmatrix}\in D(\tilde\ca)
\end{equation}
for all $\begin{pmatrix} x \\ y \end{pmatrix}\in\cx$. To this end, we observe that the assumption $\rg(V)\subset\rg(U)$ implies $Ux+Vy\in\rg(U)=D(\tilde A)$. Moreover, for arbitrary $y\in Y$ the element $-U^{-1}Vy\in X$ is well defined and
\begin{equation}\label{eq:4.4}
\ca^{-1}\begin{pmatrix} -U^{-1}Vy \\ y \end{pmatrix}=\begin{pmatrix} 0 \\ (-WU^{-1}V+S)y \end{pmatrix}\in D(\ca)
\end{equation}
gives $(-WU^{-1}V+S)y=(LV+S)y\in D(D)$. Using these facts we obtain
\[
L[Ux+Vy]+[Wx+Sy]=(LV+S)y-WU^{-1}Ux+Wx=(LV+S)y\in D(D)
\]
proving \eqref{eq:4.3}.

At this point it only remains to prove invertibility of operator $D$ and to verify that the operators $\tilde A$, $B$, $D$ and $L$ satisfy our general assumptions. To this end, we first note that $\tilde A=U^{-1}$ is invertible, hence closed and that $L=-W\tilde A:[D(\tilde A)]\to Y$ is bounded. Next, by the first equality in \eqref{eq:4.2} it follows that $B=-U^{-1}V D$. Here the operator $U^{-1}V:Y\to X$ is bounded since $\rg(V)\subseteq \rg(U)$ by assumption. The second equality in \eqref{eq:4.2} then implies $(S-WU^{-1}V)Dy=y$ for all $y\in D(D)$. On the other hand, multiplying \eqref{eq:4.4} from the left by $\ca$ gives $D(S-WU^{-1}V)y=y$ for all $y\in Y$. This shows that $D$ is invertible, hence closed, and that $B=-U^{-1}V D:[D(D)]\to X$ is bounded as claimed.
\end{proof}

As a consequence we obtain a large class of osc operator matrices.

\begin{corollary}\label{cor:4.2}
Assume that $\dim X<\infty$. If ${\ca}$ is the generator of a strongly continuous semigroup on $\cx$, then ${\ca}-\lambda$ is an osc operator matrix for $\lambda$ large enough.
\end{corollary}

\begin{proof}
By the Hille–Yosida Theorem, $\lambda R(\lambda,{\ca})$ exists for $\lambda$ large and converges strongly to $I_{\cx}$ as $\lambda\to+\infty$. In particular
\[
\lambda R(\lambda,{\ca})\begin{pmatrix} x \\ 0 \end{pmatrix}=: \lambda\begin{pmatrix} U_\lambda & V_\lambda \\ W_\lambda & S_\lambda \end{pmatrix}\begin{pmatrix} x \\ 0 \end{pmatrix}= \begin{pmatrix} \lambda U_\lambda x \\ \lambda W_\lambda x \end{pmatrix} \rightarrow\begin{pmatrix} x \\ 0 \end{pmatrix}
\]
for all $x\in X$. Since $X$ is finite dimensional, this implies $\lambda U_\lambda\rightarrow I_X$ with respect to the operator norm. Hence $U_\lambda$ is invertible for $\lambda$ large enough and Proposition~\ref{prop:4.1}(b) gives the assertion.
\end{proof}

\begin{examples}\label{ex:4.3}
Let us consider again the operators associated to the equations from \eqref{sec:volterra-delay}.

\begin{enumerate}
\item (VIDE). For every $\lambda\in\rho\left(\dds\right)$ one can easily see that
\begin{equation}\label{eq:4.5}
\ca - \lambda = \begin{pmatrix}
A -\lambda & \delta_0 \\
C(\cdot) & \dds - \lambda
\end{pmatrix}
= \begin{pmatrix}
\tilde A_\lambda-\lambda & \delta_0 \\
0 & D-\lambda
\end{pmatrix}
\begin{pmatrix}
I_X & 0 \\
L_\lambda & I_{F(\pos,X)}
\end{pmatrix},
\end{equation}
where $\tilde A_\lambda:=A+\delta_0R\left(\lambda, \dds\right)C(\cdot)$, $L_\lambda:=-R(\lambda,{d\over ds})C(\cdot)$, and $D={d\over ds}$ with domain depending on the choice of the function space $F$ (see \cite[I.4.16]{EN00}). Hence, $\ca - \lambda$ is an osc operator matrix.

\item (ADDE). The matrix ${\ca}$ introduced in \eqref{eq:ADDE} is an osc matrix on ${\mathcal X}:=X\times L^p([-1,0],X)$. In fact, we can decompose it as
\begin{equation}\label{eq:4.6}
\ca = \begin{pmatrix}
\tilde A & \Phi \\
0 & D
\end{pmatrix}
\begin{pmatrix}
I_X & 0 \\
L & I_{L^p}
\end{pmatrix},
\end{equation}
where $\tilde A:=A-\Phi L$, $D={d\over ds}$ with domain $D(D)=\{f\in W^{1,p}([-1,0],X):f(0)=0\}$, and $L:= - |1\otimes I_X$, i.e., $Lx$ is the function of constant value $-x$. By \eqref{eq:2.2} we have
\[
\begin{pmatrix} x \\ f \end{pmatrix}\in D(\ca) \iff x \in D(A),\quad f\in W^{1,p}([-1,0],X) \quad \text{and} \quad f(0)=x.
\]
Since the latter condition is equivalent to $f-|1\otimes x = Lx + f \in D(D)$, the domain of ${\ca}$ according to Definition~\ref{def:osc-matrix} actually coincides with the domain introduced in \eqref{sec:volterra-delay} for \eqref{eq:ADDE}. Moreover, for $\begin{pmatrix} x \\ f \end{pmatrix}\in D(\ca)$ we obtain
\[
\begin{aligned}
\begin{pmatrix} \tilde A & \Phi \\ 0 & D \end{pmatrix}\begin{pmatrix} I_X & 0 \\ L & I_Y \end{pmatrix}\begin{pmatrix} x \\ f \end{pmatrix}
&=\begin{pmatrix} A-\Phi L & \Phi \\ 0 & {d\over ds} \end{pmatrix}\begin{pmatrix} x \\ Lx+f \end{pmatrix} \\
&=\begin{pmatrix} Ax-\Phi Lx+\Phi Lx+\Phi f \\ (Lx+f)' \end{pmatrix}=\begin{pmatrix} Ax+\Phi f \\ f' \end{pmatrix}={\ca}\begin{pmatrix} x \\ f \end{pmatrix}
\end{aligned}
\]
proving the representation of $\ca$ as in \eqref{eq:4.6}.

We remark that, in this case, the operator $D$ defined in Proposition~\ref{prop:4.1} equals $\dds$ with domain
\[
D\left(\dds\right)= \left\{ f\in L^{p}\left([-1,0],X\right): \begin{pmatrix} 0 \\ f \end{pmatrix}\in D(\ca)\right\} =\left\{ f\in W^{1,p}\left([-1,0],X\right): f(0)=0 \right\}.
\]
Therefore, $D$ is the generator of the nilpotent (left) translation semigroup (see \cite[I.4.17]{EN00}) and is invertible, since $\sigma(D)=\emptyset$. Note that we cannot use Proposition~\ref{prop:4.1} in order to show that $\ca$ is osc, since in general it is not invertible.

\item Let us now give an example of a \textit{non}-osc operator matrix. Take $\cx:= L^p(0,1)\times L^p(0,1)$ and $\tilde\ca$ defined by
\[
\begin{aligned}
D(\tilde\ca) &:= \left\{ \begin{pmatrix} f \\ g \end{pmatrix} \in W^{1,p}(0,1)\times W^{1,p}(0,1) : f(0)=g(0), f(1)=g(1) \right\}, \\
\tilde\ca &:=\begin{pmatrix} \dds & 0 \\ 0 & -\dds \end{pmatrix}
\end{aligned}
\]
which can be identified with the generator of the `counter-clockwise' shift (semi)group of isometries on the circle. Then the operator $\ca:=\tilde\ca-I_\cx$ is invertible since it generates a uniformly exponentially stable semigroup. Moreover, the operator $D$ defined in Proposition~\ref{prop:4.1} is given by
\[
\begin{aligned}
D(D)&=\left\{ g\in L^p(0,1): \begin{pmatrix} 0 \\ g \end{pmatrix}\in D(\ca)\right\} = \left\{g\in W^{1,p}(0,1): g(0)=g(1)=0\right\}, \\
Dg&=\pi_Y\left[\ca\begin{pmatrix} 0 \\ g \end{pmatrix}\right] = -\left(\dds+I_{L^p(0,1)}\right) g.
\end{aligned}
\]
Observe that $\sigma(D)=\mathbb{C}$, hence $D$ is not invertible and, by Proposition~\ref{prop:4.1}, $\ca$ is not osc.
\end{enumerate}
\end{examples}

\section{SPECTRAL THEORY FOR ONE-SIDED COUPLED OPERATOR MATRICES}
\label{sec:spectral-theory}

We now look for a connection between the spectrum of an osc matrix $\ca$ on $\cx=X\times Y$ and the spectra of certain associated operators on the factor spaces $X$ and $Y$. Let
\[
\ca:=\begin{pmatrix}
\tilde{A} & B \\
0 & D
\end{pmatrix}
\begin{pmatrix}
I_X & 0 \\
L & I_Y
\end{pmatrix}
\]
be an osc operator matrix and for all $\lambda\in\rho(D)$ define
\[
A_{\lambda} := \tilde{A} + \lambda B R(\lambda, D) L, \qquad D(A_\lambda):=D(\tilde A).
\]

\begin{theorem}\label{thm:5.1}
With the above notation and for every $\lambda\in\rho(D)$ we have
\[
\lambda\in\rho(\ca) \iff \lambda\in\rho(A_{\lambda}).
\]
In this case
\begin{equation}\label{eq:5.1}
R(\lambda, \ca) = \begin{pmatrix}
R(\lambda, A_{\lambda}) & R(\lambda, A_{\lambda})B R(\lambda, D) \\
-L_{\lambda} R(\lambda, A_{\lambda}) & R(\lambda, D)-L_{\lambda} R(\lambda, A_{\lambda}) B R(\lambda, D)\end{pmatrix}
\end{equation}
where $L_{\lambda} := -D R(\lambda, D)L$.
\end{theorem}

\begin{proof}
Let $\lambda\in\rho(D)$. It can be verified that
\[
\lambda - \ca = \begin{pmatrix} I_X & -BR(\lambda, D) \\ 0 & I_Y \end{pmatrix}
\begin{pmatrix} \lambda - A_{\lambda} & 0 \\ 0 & \lambda - D \end{pmatrix}
\begin{pmatrix} I_X & 0 \\ L_{\lambda} & I_Y \end{pmatrix}=: {\mathcal L}_{\lambda}{\mathcal Q}_{\lambda}{\mathcal R}_{\lambda}.
\]
Observe now that ${\mathcal L}_{\lambda}$ is invertible on $X\times Y$ and ${\mathcal R}_{\lambda}$ is invertible and bounded on $[D(\tilde{A})]\times Y$, while ${\mathcal Q}_{\lambda}$ is invertible if and only if $\lambda - A_{\lambda}$ is invertible. Hence we conclude that $\lambda-\ca$ is invertible if and only if $\lambda-A_\lambda$ is invertible.

To obtain the resolvent we compute
\[
\begin{aligned}
R(\lambda, \ca) &= {\mathcal R}_{\lambda}^{-1}{\mathcal Q}_{\lambda}^{-1}{\mathcal L}_{\lambda}^{-1} \\
&= \begin{pmatrix} I_X & 0 \\ -L_{\lambda} & I_Y \end{pmatrix}
\begin{pmatrix} R(\lambda,A_{\lambda}) & 0 \\ 0 & R(\lambda,D) \end{pmatrix}
\begin{pmatrix} I_X & BR(\lambda, D) \\ 0 & I_Y \end{pmatrix} \\
&= \begin{pmatrix}
R(\lambda, A_{\lambda}) & R(\lambda, A_{\lambda})B R(\lambda, D) \\
-L_{\lambda} R(\lambda, A_{\lambda}) & R(\lambda, D)-L_{\lambda} R(\lambda, A_{\lambda}) B R(\lambda, D)
\end{pmatrix}.
\end{aligned}
\]
\end{proof}

\begin{remark}
\label{rem:5.2}
If $\dim X <\infty$, then for $\lambda\in\rho(D)$ we have
\[
\lambda\in\sigma({\ca}) \iff \lambda\in P\sigma(A_{\lambda}) \iff \det(\lambda - A_{\lambda}) = 0,
\]
i.e., we obtain a \textit{characteristic equation} for the spectral values of $\ca$, cf. \cite{Na97,En99}.
\end{remark}

\begin{examples}\label{ex:5.3}
With Theorem~\ref{thm:5.1} in mind we can now examine the spectra of the operator matrices defined in \eqref{eq:2.1} and \eqref{eq:2.2}.

\begin{enumerate}
\item (VIDE). Using the representation \eqref{eq:4.5} of $\ca-\lambda$, we obtain by Theorem~\ref{thm:5.1} for $\lambda\in\rho\left(\dds\right)$ the characterization
\[
\lambda\in\sigma(\ca) \iff \lambda \in \sigma \left(A + \delta_0 R\left(\lambda, \dds\right) C(\cdot)\right).
\]
For $\operatorname{Re}\lambda > \omega\left(\dds\right)$, using the integral representation of the resolvent of $A$, this is equivalent to
\[
\lambda\in\sigma\left( A + \int_0^{\infty} e^{-\lambda s} C(s)\; ds\right).
\]
If $X$ is finite dimensional (typically, $X=\mathbb{C}^n$), by Remark~\ref{rem:5.2} we obtain the characteristic equation
\[
\lambda\in\sigma(\ca)\iff \det\left(\lambda-A-\int_0^{\infty} e^{-\lambda s} C(s)\;ds\right) = 0,
\]
for the spectral values of $\ca$ satisfying $\operatorname{Re}\lambda > \omega\left(\dds\right)$.

\item (ADDE). The condition $\lambda\in\rho(D)=\mathbb{C}$ is always satisfied and Theorem~\ref{thm:5.1} yields
\[
\lambda\in\sigma(\ca)\iff\lambda\in \sigma(A_\lambda)=\sigma\left(A+\Phi\left(|1\otimes I_X\right)-\lambda\Phi R\left(\lambda,\dds\right)\left(|1\otimes I_X\right)\right).
\]
In order to simplify this characterization, we show that $|1\otimes I_X-\lambda R(\lambda,\dds)(|1\otimes I_X)= \varepsilon_\lambda\otimes I_X$, where $\varepsilon_{\lambda}\otimes I_X : X\to W^{1,p}\left([-1,0],X\right)$, $\left(\left(\varepsilon_{\lambda}\otimes I_X\right)(x)\right)(s) = e^{\lambda s} x$. In fact, using the explicit formula for the resolvent of the first derivative (see \cite[II.2.10]{EN00}) we obtain for all $x\in X$ and $s\in[-1,0]$ that
\[
\begin{aligned}
[(\varepsilon_\lambda\otimes I_X)x](s) &= e^{\lambda s}x = x-\lambda\int_s^0 e^{-\lambda(\tau-s)}x\,d\tau \\
&= x-\lambda \int_s^0 e^{\lambda k}dk\,x \\
&= \left[|1\otimes I_X-\lambda R\left(\lambda,\dds\right)(|1\otimes I_X)\right]x.
\end{aligned}
\]
Therefore, we have
\[
\lambda\in\sigma(\ca) \iff \lambda \in \sigma \left(A + \Phi \left(\varepsilon_{\lambda}\otimes I_X\right)\right).
\]
Moreover, if $\dim X < \infty$, by Remark~\ref{rem:5.2} it holds that
\[
\lambda\in\sigma(\ca)\iff \det\left(\lambda - A-\Phi \left(\varepsilon_{\lambda}\otimes I_X\right)\right) = 0,
\]
i.e., the spectral values of $\ca$ are again the zeros of a characteristic equation.
\end{enumerate}
\end{examples}

\section{ONE-SIDED COUPLED GENERATORS OF STRONGLY CONTINUOUS SEMIGROUPS}
\label{sec:generators}

We have already characterized osc operator matrices and also mentioned several spectral properties they inherit from their entries. We will now try to give some conditions which imply that a given osc operator matrix generates a strongly continuous semigroup on the product space $\cx$. A motivation for this task is the fact that, as a rule, \textit{wellposedness of \eqref{eq:ACP} for the osc operator matrix $\ca$ is equivalent to wellposedness of the original initial value problem} (in a suitable sense). A proof of this claim for the particular cases of Volterra integro-differential equations and delay differential equations can be found in \cite[\S\S~VI.6--7]{EN00}, but similar results have actually been proven in a general context of abstract initial boundary value problems in \cite{Mu01} and \cite{KMN02}.

Checking the generator property for an osc operator matrix $\ca$ is particularly easy if the first factor of $\ca$ in Definition~\ref{def:osc-matrix} is a diagonal matrix, i.e., if $B=0$. The general case will then be obtained by means of similarity and perturbation arguments.

\begin{theorem}
\label{thm:6.1}
Assume that $\ca$ is an osc operator matrix such that $B=0$, i.e.,
\[
\ca=\begin{pmatrix}
\tilde{A} & 0 \\
0 & D
\end{pmatrix}
\begin{pmatrix}
I_X & 0 \\
L & I_Y
\end{pmatrix},
\]
defined on the usual domain introduced in Definition~\ref{def:osc-matrix}. Then $\ca$ is the generator of a strongly continuous semigroup $\ttt$ on $\cx$ if and only if the following conditions hold.
\begin{enumerate}[label=(\roman*)]
\item $\tilde A$ and $D$ generate strongly continuous semigroup $(T(t))_{t\ge0}$ and $(S(t))_{t\ge0}$ on $X$ and $Y$, respectively, and
\item for all $t\geq0$ the operator
\begin{equation}\label{eq:6.1}
\tilde Q(t):D(\tilde A^2)\rightarrow Y,\qquad \tilde Q(t)x:=D\int_0^t S(t-s)LT(s)x\;ds,
\end{equation}
extends to a bounded linear operator $Q(t)$ from $X$ to $Y$, and this extension satisfies
\[
\limsup\limits_{t\rightarrow0^+}\norm{Q(t)}_{\mathcal{L}(X,Y)}<\infty.
\]
\end{enumerate}
In this case, $(\ct(t))_{t\ge0}$ is given by
\begin{equation}\label{eq:6.2}
\ct(t)=\begin{pmatrix}
T(t) & 0 \\
Q(t) & S(t)
\end{pmatrix}.
\end{equation}
\end{theorem}

\begin{remark}
\label{rem:6.2}
\begin{enumerate}
\item We note that the operators $\tilde Q(t)$ are well defined. In fact, the function $\pos\ni s\mapsto LT(s)x$ is differentiable for all $x\in D(\tilde A^2)$, and this implies, as one can see by integrating by parts, that $\int_0^t S(t-s)LT(s)x\,ds\in D(D)$.
\item We also remark that there are examples of generators $\tilde A$ and $D$ such that the family of operators $\tilde Q(t)$ defined in \eqref{eq:6.1} does not satisfy the assumptions in (ii). Take, e.g., an unbounded generator $A$ of a strongly continuous group on a Banach space $X$. Then for
\[
\ca:=\begin{pmatrix}
A & 0 \\
0 & A
\end{pmatrix}
\begin{pmatrix}
I_X & 0 \\
I_X & I_X
\end{pmatrix}
\]
condition (i) of Theorem~\ref{thm:6.1} is satisfied, while (ii) is not, and hence $\ca$ is no generator on $\cx=X\times X$.
\end{enumerate}
\end{remark}

\begin{proof}[Proof of Theorem \ref{thm:6.1}]
We first assume that $\ca$ is a generator. Since $B=0$, the representation of the resolvent in \eqref{eq:5.1} becomes
\[
R(\lambda, \ca) = \begin{pmatrix}
R(\lambda, \tilde A) & 0 \\
DR(\lambda,D)LR(\lambda,\tilde A) & R(\lambda, D)
\end{pmatrix}.
\]
Hence, the powers are
\[
R(\lambda, \ca)^n = \begin{pmatrix}
R(\lambda,\tilde A)^n & 0 \\
* & R(\lambda, D)^n
\end{pmatrix}\quad\text{for all }n\in \mathbb{N}.
\]
Therefore, both $\norm{R(\lambda,\tilde A)^n}$ and $\norm{R(\lambda,D)^n}$ are dominated by $\norm{R(\lambda,\ca)^n}$, which by the general version of the Hille–Yosida Theorem (see \cite[Thm.~II.3.8]{EN00}) satisfies an estimate of the form
\[
\norm{R(\lambda,\ca)^n} \leq \frac{M}{(\operatorname{Re}\lambda-\omega)^n}
\]
for all $\lambda\in\mathbb{C}$ such that $\operatorname{Re}\lambda>\omega$ for some $\omega$ sufficiently large. Hence, in order to apply the Hille–Yosida Theorem to $\tilde A$ and $D$ we only need that both operators are densely defined. This follows for $\tilde A$ from the fact that $D(\ca)\subseteq D(\tilde A)\times Y$ is dense in $\cx=X\times Y$, hence $D(\tilde A)$ must be dense in $X$. To show the denseness of $D(D)$ in $Y$, we fix some $f\in Y$ and consider
\[
\lambda R(\lambda,\ca)\begin{pmatrix} 0 \\ f \end{pmatrix}=\begin{pmatrix} 0 \\ \lambda R(\lambda,D)f \end{pmatrix}\to \begin{pmatrix} 0 \\ f \end{pmatrix}
\]
as $\lambda\to+\infty$. Since $\lambda R(\lambda,D)f\in D(D)$, this shows that $D(D)$ is dense in $Y$ and thus $\tilde A$ and $D$ are generators of strongly continuous semigroups $(T(t))_{t\ge0}$ and $(S(t))_{t\ge0}$ on $X$ and $Y$, respectively.

In order to prove the claim on the operators $Q(t)$ we note that for $\lambda$ with real part large enough the Laplace transform of the semigroup $\ct(\cdot)$ coincides with the resolvent $R(\lambda,\ca)$ of its generator. In particular, we obtain for the Laplace transform of the lower left entry ${\mathcal T}(\cdot)_{21}$ of $\ct(\cdot)$ the equality
\[
\mathcal{L}({\mathcal T}(\cdot)_{21})(\lambda)= R(\lambda,\ca)_{21}=DR(\lambda,D)LR(\lambda,\tilde A).
\]
On the other hand, the convolution theorem for the Laplace transform implies that
\[
\mathcal{L}(\tilde Q(\cdot)x)(\lambda) = DR(\lambda,D)LR(\lambda,\tilde A)x
\]
for all $x\in D(\tilde A^2)$ and $\operatorname{Re}\lambda$ sufficiently large. This proves that $\mathcal{L}(\tilde Q(\cdot)x)(\lambda)=\mathcal{L}({\mathcal T}(\cdot)_{21}x)(\lambda)$ for all $x\in D(\tilde A^2)$ and $\operatorname{Re}\lambda$ large. The denseness of the domain $D(\tilde A^2)$ in $X$ (see \cite[Prop.~II.1.8]{EN00}) and the injectivity of the Laplace transformation then imply that $\ct(t)_{21}=Q(t)$. The remaining assertion on the boundedness of $\|Q(t)\|$ close to $t=0$ then follows from the boundedness of $(\ct(t))_{t\ge0}$ as $t\to0^+$.

To show the converse implication, let us first define $\ct(t)$ as in \eqref{eq:6.2}. We are going to show that the operator family $(\ct(t))_{t\geq0}$ is a strongly continuous semigroup and then verify that its generator is $\ca$. In fact, for $t,s\in\pos$, and $x\in D(\tilde A^2)$, $y\in Y$ we have
\[
\begin{aligned}
\ct(t)\ct(s)\begin{pmatrix} x \\ y \end{pmatrix} &=\begin{pmatrix} T(t) & 0 \\ Q(t) & S(t) \end{pmatrix}
\begin{pmatrix} T(s) & 0 \\ Q(s) & S(s) \end{pmatrix}\begin{pmatrix} x \\ y \end{pmatrix} \\
&= \begin{pmatrix} T(t) & 0 \\ Q(t) & S(t) \end{pmatrix}
\begin{pmatrix} T(s)x \\ D\int_0^s S(s-r)LT(r)x\;dr+S(s)y \end{pmatrix} \\
&= \begin{pmatrix} T(t)T(s)x \\ D\int_0^t S(t-q)LT(q)T(s)x\;dq+S(t)D\int_0^s S(s-r)L T(r)x\;dr+S(t)S(s)y \end{pmatrix} \\
&= \begin{pmatrix} T(t+s)x \\ D\int_s^{t+s} S(t+s-r)LT(r)x\;dr+ D\int_0^s S(t+s-r)LT(r)x\;dr+S(t+s)y \end{pmatrix} \\
&= \begin{pmatrix} T(t+s)x \\ D\int_0^{t+s} S(t+s-r)LT(r)x\;dr+S(t+s)y \end{pmatrix} \\
&= \ct(t+s)\begin{pmatrix} x \\ y \end{pmatrix}.
\end{aligned}
\]
From \cite[Prop.~II.1.8]{EN00} it follows that $D(\tilde A^2)\times Y\subseteq \cx$ is dense. Since the operators $\ct(t)$ and $\ct(s)$ are bounded, this proves the semigroup property $\ct(t)\ct(s)=\ct(t+s)$. Moreover,
\[
\ct(0)=\begin{pmatrix} T(0) & 0 \\ Q(0) & S(0) \end{pmatrix}=\begin{pmatrix} I_X & 0 \\ 0 & I_Y \end{pmatrix}=I_\cx,
\]
since $Q(0)x=\tilde Q(0)x=0$ for all $x\in D(\tilde A^2)$ and hence $Q(0)=0$ by the denseness of $D(\tilde A^2)$ in $X$.

To show strong continuity of $(\ct(t))_{t\ge0}$, i.e.,
\[
\lim_{t\rightarrow0^+}\ct(t)\begin{pmatrix} x \\ y \end{pmatrix} = \begin{pmatrix} \lim\limits_{t\rightarrow0^+}T(t)x \\ \lim\limits_{t\rightarrow0^+}\Bigl[Q(t)x+S(t)y\Bigr] \end{pmatrix}= \begin{pmatrix} x \\ y \end{pmatrix}
\]
for all $\begin{pmatrix} x \\ y \end{pmatrix}\in\cx$, it suffices to verify that operators $Q(t)$ converge to $0$ as $t\to0^+$. To this end we fix some $\lambda\in\rho(D)\cap\rho(\tilde A)$ and choose $x\in D(\tilde A^2)$. Then, integrating by parts and using the fact that $(e^{-\lambda t}T(t))_{t\ge0}$ and $(e^{-\lambda t}S(t))_{t\ge0}$ are strongly continuous semigroups with generators $\tilde A-\lambda$ and $D-\lambda$, respectively, we obtain
\[
\begin{aligned}
Q(t)x &= e^{\lambda t}D\int_0^t e^{-(t-s)\lambda}S(t-s)Le^{-\lambda s}T(s)x\;ds \\
&= e^{\lambda t}D(D-\lambda)^{-1}
\bigg(\Big[-e^{-(t-s)\lambda}S(t-s)Le^{-\lambda s}T(s)x\Big]^{s=t}_{s=0} \\
&\quad +\int_0^t e^{-(t-s)\lambda}S(t-s)Le^{-\lambda s}T(s)({\tilde A}-\lambda)x\;ds\bigg) \\
&= D(D-\lambda)^{-1}\Big[ S(t)Lx-LT(t)x+\int_0^t S(t-s)L(\tilde A-\lambda)^{-1}T(s)(\tilde A-\lambda)^2x\;ds\Big] \\
&= D(D-\lambda)^{-1}\Big[ S(t)Lx-L(\lambda-A)^{-1}T(t)(\lambda-A)x \\
&\quad +\int_0^t S(t-s)L(\tilde A-\lambda)^{-1}T(s)(\tilde A-\lambda)^2x\;ds\Big].
\end{aligned}
\]
Since $LR(\lambda,\tilde A)$ and $DR(\lambda,D)$ are bounded operators,
\begin{equation}\label{eq:6.3}
\begin{split}
\norm{Q(t)x} &\leq \norm{D(D-\lambda)^{-1}}\cdot\Big[\underbrace{\norm{S(t)Lx-L(\lambda-A)^{-1}T(t)(\lambda-A)x}}_{\rightarrow0\text{ as }t\to0^+}\\
&\qquad + \underbrace{\int_0^t\norm{S(t-s)L(\tilde A-\lambda)^{-1}T(s)}\cdot\norm{(\tilde A-\lambda)^2x} ds}_{\rightarrow0\text{ as }t\to0^+}\Big]. 
\end{split}
\end{equation}
This shows that $\lim\limits_{t\to0^+}Q(t)=0$ on $D(\tilde A^2)$. The assumption on the boundedness of $\norm{Q(t)}$ as $t\rightarrow 0^+$ and the density of $D(\tilde A^2)$ in $X$ then yield $\lim\limits_{t\to0^+}Q(t)=0$ on the whole space $X$.

Finally, we have to show that $\ca$ is actually the generator of $(\ct(t))_{t\geq0}$. To this end it suffices to note that, as already mentioned above, the Laplace transform of $Q(\cdot)$ for $\operatorname{Re}\lambda$ sufficiently large coincides with the Laplace transform of the lower left entry of the resolvent $R(\lambda,\ca)$.
\end{proof}

\begin{proposition}\label{prop:6.3}
Assume that $\tilde A$ is bounded on $X$. Then an osc operator matrix $\ca$ generates a strongly continuous (resp., analytic) semigroup on $\cx$ if and only if the operator $D+LB$, defined on $D(D)$, generates a strongly continuous (resp., analytic) semigroup on $Y$.
\end{proposition}

\begin{proof}
We begin by remarking that due to the boundedness of $\tilde A$, and by the assumption (A$_3$), $L$ is bounded and hence the matrix
\[
\cs:=\begin{pmatrix} I_X & 0 \\ L & I_Y \end{pmatrix}
\]
is bounded and invertible on $\cx$. Therefore, $\ca$ is similar to
\[
{}^{s\!\!\!}{\ca}:=\begin{pmatrix} I_X & 0 \\ L & I_Y \end{pmatrix}\begin{pmatrix} \tilde{A} & B \\ 0 & D \end{pmatrix} = \cs\ca \cs^{-1}.
\]
However,
\[
{}^{s\!\!\!}{\ca}=\begin{pmatrix} \tilde A & B \\ L\tilde A & D+LB \end{pmatrix}=\begin{pmatrix} 0 & B \\ 0 & D+LB \end{pmatrix}+\begin{pmatrix} \tilde A & 0 \\ L\tilde A & 0 \end{pmatrix}=:\tilde\ca+\tilde\cb,
\]
where $\tilde\cb$ is bounded and $D(\tilde\ca)=X\times D(D)$. Hence, by bounded perturbation and similarity arguments, we can conclude that $\ca$ is a generator if and only if the matrix operator $\tilde\ca$ is a generator. By inverting in Theorem~\ref{thm:6.1} the order of the spaces $X$ and $Y$ (obtaining in this way a generation result for upper instead of lower triangular matrices, see also \cite[Cor.~3.2 and Cor.~3.3]{Na89}), one can see that this is the case if and only if $D+LB$ is a generator on $Y$.
\end{proof}

\begin{example}\label{ex:6.4}
We now apply the above results to delay differential equations. Let $A$ be bounded. Then, due to the assumptions on the various operators, also $\tilde A=A-\Phi L$ is bounded and Proposition~\ref{prop:6.3} applies. Hence, the osc operator matrix $\ca$ associated to \eqref{eq:ADDE} generates a strongly continuous semigroup on ${\mathcal X}$ if and only if
\[
D+L\Phi={d\over ds}-(|1\otimes I_X)\Phi\quad\text{defined on}\quad D\left(\frac{d}{ds}\right)=\left\{ f\in W^{1,p}\left([-1,0],X\right) : f(0)=0\right\}
\]
generates a strongly continuous semigroup on $L^p([-1,0],X)$. In particular, if $\dim X<\infty$, by the Desch–Schappacher perturbation theorem one can prove that ${d\over ds}+(|1\otimes I_X)\Phi$ is always a generator, hence \eqref{eq:ADDE} is always well-posed. This applies, for instance, if $\Phi$ is a linear combination of Dirac measures on the interval $[-1,0]$. For the details and further results see \cite[Example~3.16]{En99} and \cite{En97b}. We also emphasize that a different approach is needed when $\dim X=\infty$ and $\Phi$ is unbounded. We refer the reader to \cite{BP02} for further reading on this topic.
\end{example}

\begin{remark}\label{rem:6.5}
It is also possible to find conditions such that the osc operator matrix associated to \eqref{eq:VIDE} on $\cx=X\times C_0(\pos,X)$ and on $\cx=X\times L^p(\pos,X)$, respectively, generates a strongly continuous semigroup, see \cite{En99}.
\end{remark}

\section{ONE-SIDED COUPLED OPERATOR MATRICES AND POSITIVITY}
\label{sec:positivity}

We now consider positive semigroups of operators on Banach lattices, and refer to \cite{Na86} for a thorough discussion of this topic. We only recall that a strongly continuous semigroup is positive if and only if the resolvent of its generator is positive for sufficiently large $\lambda$ (see \cite[Thm.~VI.1.8]{EN00}). However, we would like to use our ``matrix decomposition'' of osc operators to give a characterization in terms of the matrix entries. In this section (and later on) whenever referring to positivity we assume $X$ and $Y$ to be Banach lattices.

\begin{theorem}\label{thm:7.1}
Under the Assumptions~\ref{assumptions:3.1} assume that the osc operator matrix
\begin{equation}\label{eq:7.1}
\ca:=\begin{pmatrix}
\tilde{A} & B \\
0 & D
\end{pmatrix}
\begin{pmatrix}
I_X & 0 \\
L & I_Y
\end{pmatrix},
\end{equation}
with $L\in \mathcal{L}(X,Y)$, generates a semigroup $\ttt$ and let $A_{\lambda} := \tilde{A} + \lambda B R(\lambda, D) L$ for $\lambda\in\rho(D)$. Then the following properties are equivalent.
\begin{enumerate}[label=(\alph*)]
\item $\ttt$ is positive.
\item There exists $\lambda_0 \in \mathbb{R}$ such that
\begin{enumerate}[label=(\roman*)]
\item $R(\lambda, A_{\lambda}) \ge 0$ for all $\lambda > \lambda_0$,
\item $R(\lambda, D) \ge 0$ for all $\lambda > \lambda_0$,
\item $By \ge 0$ for all $y\in D(D)_{+}$ and $L^*D^*y^* \ge 0$ for all $y^*\in D(D^*)_{+}$.
\end{enumerate}
\end{enumerate}
\end{theorem}

\begin{remark}\label{rem:7.2}
\begin{enumerate}
\item By $D^*$ and $L^*$ in (b.iii) we denote the adjoints of the operators $D$ and $L$. Observe that the operator $DL$ (which is the lower left entry of the matrix one obtains by directly computing the product in \eqref{eq:7.1}) does not make sense in general (or, more precisely, may have domain $\{0\}$ as in the case of the matrix associated to the delay equation \eqref{eq:ADDE}), while the operator $L^*D^*$ has domain $D(D^*)$ since $L$ is bounded.
\item One can show (see \cite[Remark~3.5]{En97a}) that the assumption (b.iii) is equivalent to
\begin{quote}
(iii') $By \ge 0$ for all $y\in D(D)_{+}$ and $-L_\lambda\ge0$ for all $\lambda>\lambda_0$,
\end{quote}
where $L_\lambda:=-DR(\lambda,D)L$, avoiding in this way the use of the adjoint operators $D^*$ and $L^*$.
\end{enumerate}
\end{remark}

\begin{proof}[Sketch of proof]
We only give the main steps and refer the reader to \cite[Thm.~3.2]{En97a} for the details.

(b)$\Rightarrow$(a). Recalling the explicit formula \eqref{eq:5.1} for resolvent of $\ca$, it suffices to observe that under assumption (b) (with (iii) replaced by (iii'), see Remark~\ref{rem:7.2}(2)) it holds that $R(\lambda,\ca)\ge 0$ for all $\lambda > \lambda_0$. Therefore $\ttt$ is positive by \cite[Theorem VI.1.8]{EN00}.

(a)$\Rightarrow$(b). Assume $\ttt$ to be positive. Then there exists a suitable $\lambda_0$ such that $R(\lambda,\ca)\ge 0$ for all $\lambda > \lambda_0$. This implies that each entry of the resolvent operator matrix must be positive (to check this, it suffices to take $x\in X$, $y\in Y$, and then consider $R(\lambda,\ca)\begin{pmatrix} x \\ 0 \end{pmatrix}$ and $R(\lambda,\ca)\begin{pmatrix} 0 \\ y \end{pmatrix}$). Hence, in particular the upper-left entry $R(\lambda,A_\lambda)$ is positive for $\lambda$ large enough proving (b.i).

(b.iii) follows from the \textit{positive minimum principle} applied to $\ca$ and its adjoint $\ca^*$, see \cite[Lemma~2.7]{En97a} for details.

In order to prove (b.ii) consider
\[
\ca_v = \ca - \begin{pmatrix} v I_X & 0 \\ 0 & 0 \end{pmatrix}, \quad v\ge 0,
\]
and observe that with $\ca$ also $\ca_v$ generates a positive semigroup. Hence, the resolvent $R(\lambda, \ca_v)$ is positive for large $\lambda$ and, by some further calculations, one obtains for the lower right entry of $R(\lambda, \ca_v)$
\[
0\le \lim_{v\to\infty} R(\lambda,\ca_v)_{22}=R(\lambda, D),
\]
which shows (ii).
\end{proof}

One important reason to consider positive semigroups is the fact that positivity greatly simplifies the study of the asymptotic behavior. Recall that the \textit{spectral bound} of a closed operator $A$ is defined as
\[
s(A) := \sup\{ \operatorname{Re} \lambda : \lambda \in\sigma(A)\}.
\]
In case $A$ is the generator of a strongly continuous semigroup $\left(T(t)\right)_{t\ge 0}$, it is always dominated by the \textit{growth bound}
\[
\omega_0 := \inf\left\{ w\in \mathbb{R} : \exists M_w\ge 1 : \;\norm{T(t)} \leq M_w e^{wt}\enspace\forall\, t\ge 0\right\}.
\]
These two bounds differ in general, but are equal in the special case of positive semigroups on $L^p$ spaces (cf. \cite[Thm.~VI.1.15]{EN00}) and of analytic semigroups (cf. \cite[Cor.~IV.3.12]{EN00}). Also, it is usually difficult to determine $s(A)$. However, if $A$ is self-adjoint, or if it generates a positive semigroup, then $s(A)\in\sigma(A)$ (cf. \cite[Thm.~VI.1.10]{EN00}) and hence it suffices to calculate the real spectral values of $A$ in order to determine $s(A)$. Another nice property says that a positive strongly continuous semigroup is exponentially stable if and only if the spectral bound of its generator is strictly negative (cf. \cite[Prop.~VI.1.14]{EN00}). Combining the general theory of positive semigroups with the above results, one obtains the following characterization.

\begin{corollary}\label{cor:7.3}
A positive strongly continuous semigroup $\ttt$ generated by an osc operator matrix $\ca$ (as in Theorem~\ref{thm:7.1}) is exponentially stable if and only if the spectral bounds $s(\tilde A)$ and $s(D)$ satisfy
\[
s(\tilde A)<0\quad\text{and}\quad s(D)<0.
\]
\end{corollary}

For a proof we refer to \cite[Thm.~4.1]{En97a}.

\begin{examples}\label{ex:7.4}
\begin{enumerate}
\item (ADDE). Let $\ttt$ be the semigroup generated by the operator $\ca$ defined in \eqref{eq:ADDE}. Keeping in mind that, for $D=\dds$ defined on $W^{1,p}([-1,0],X)$ we have $s(D)=-\infty$, and that $R(\lambda,D)$ is positive for all $\lambda\in\mathbb{C}$, Theorem~\ref{thm:7.1} reads as follows in this case.

\begin{quote}
\itshape $\ttt$ is positive if and only if the following two conditions hold.
\begin{enumerate}[label=(\roman*)]
\item There exists $\lambda_0\in \mathbb{R}$ such that $R(\lambda,A+\Phi(\varepsilon_{\lambda}\otimes I_X))\ge 0$ for all $\lambda\geq\lambda_0$,
\item $\Phi f \ge 0$ for all $f\in D(D)_+ = \lbrace f\in W^{1,p}([-1,0],X): f\ge0,\, f(0)=0 \rbrace$.
\end{enumerate}
\end{quote}

It is possible to show that both conditions are satisfied if
\[
(A\delta_0+\Phi)f\ge0\quad\text{for all}\quad0\le f\in W^{1,p}([-1,0],X).
\]
From Corollary~\ref{cor:7.3} we obtain the following stability result for \eqref{eq:ADDE}.

\begin{quote}
If $\ttt$ is positive, then $\ttt$ is exponentially stable if and only if $s(A+\Phi(|1\otimes I_X))<0$.
\end{quote}

\item Consider now a special case of \eqref{eq:ADDE} where we take $\Phi = \Psi\delta_{-r}$ for some $r\in [0,1]$ and $\Psi \in\mathcal{L}(X)$. Then the delay differential equation becomes
\[
\begin{cases}
\dot{u}(t) = Au(t) + \Psi u(t-r), & t \geq 0, \\
u(s) = h(s), & s\in [-1,0].
\end{cases}
\tag{*}\label{eq:ADDE-star}
\]
Assume that the associated osc operator matrix $\ca$ generates a positive semigroup. Then the solution of this equation is exponentially stable if and only if
\[
s\left(A+ \Psi\delta_{-r}(|1 \otimes I_X)\right)= s(A+\Psi) <0.
\]
This is exactly the case when the solutions of the \textit{undelayed} equation
\[
\begin{cases}
\dot{u}(t) = \left(A+\Psi\right)u(t), & t \geq 0, \\
u(s) = h(s), & s\in [-1,0]
\end{cases}
\]
are exponentially stable! Hence,
\begin{quote}
\itshape the stability of the solutions of \eqref{eq:ADDE-star} is independent of the delay $r$.
\end{quote}
This result deeply relies on the positivity assumptions and is in sharp contrast to the general (nonpositive) situation (see \cite{BP02}).
\end{enumerate}
\end{examples}

\begin{remark}\label{rem:7.5}
It is possible to find conditions such that the osc operator matrix associated to \eqref{eq:VIDE} on $\cx=X\times C_0(\pos,X)$ or on $\cx=X\times L^p(\pos,X)$, respectively, generates a positive strongly continuous semigroup. Under suitable assumptions, stability properties can also be discussed, see \cite{En02}.
\end{remark}

\section{APPLICATION TO INITIAL--BOUNDARY VALUE PROBLEMS}
\label{sec:application-ibvp}

In the final part of these notes we discuss an application of the theory of osc operator matrices to abstract boundary feedback systems. To that purpose, we consider
\begin{itemize}
\item a linear operator $(A,D(A))$ on the Banach space $\dx$,
\item a linear operator $(D_m,D(D_m))$ on the Banach space $X$, called \textit{maximal operator},
\item a linear \textit{feedback operator} $(B,D(B))$ from $X$ to $\dx$ with $D(D_m)\subseteq D(B)$,
\item a linear \textit{boundary operator} $\cp:D(D_m)\rightarrow\partial X$.
\end{itemize}

With these objects we call
\[
\begin{cases}
\dot{x}(t) = Ax(t)+Bu(t), & t\geq0, \\
\dot{u}(t) = D_m u(t), & t\geq0, \\
x(t) = \cp u(t), & t\geq0, \\
x(0) = x_0, \\
u(0) = u_0,
\end{cases}
\tag{AIBVP}\label{eq:AIBVP}
\]
an \textit{abstract initial--boundary value problem on the state space $X$ and the boundary space $\dx$}, cf. \cite{KMN02}, where also inhomogeneous problems are treated by means of osc theory. In order to apply the theory of osc operator matrices, we impose the following (cf. \cite{CENN02}).

\begin{definition}[Assumptions 8.1]
\label{assumptions:8.1}
\begin{enumerate}[label=(\alph*)]
\item The operator $A$ generates a strongly continuous semigroup $(T(t))_{t\geq0}$ on $\partial X$.
\item The restriction $D:=D_m|_{\ker(\cp)}$ generates a strongly continuous semigroup $(S(t))_{t\geq0}$ on $X$.
\item The boundary operator $\cp$ is surjective.
\item The operator $\begin{pmatrix} \cp \\ D_m \end{pmatrix}:D(D_m)\subseteq X \rightarrow X\times\dx$ is closed.
\end{enumerate}
\end{definition}

As before, our aim is to rewrite \eqref{eq:AIBVP} as an abstract Cauchy problem \eqref{eq:ACP} on the product space $\cx:=\dx\times X$. The following key lemma is essentially due to Greiner (\cite[Lemma 1.2]{Gr87}, see also \cite[Lemma~2.2]{CENN02}).

\begin{lemma}\label{lem:8.2}
Let $\lambda\in\rho(D)$. Then the restriction $\cp|_{\ker(\lambda-D_m)}$ is invertible and its inverse
\[
\cd_\lambda:=(\cp|_{\ker(\lambda-D_m)})^{-1}:\partial X\rightarrow\ker(\lambda-D_m)\subseteq X,
\]
called the associated \textit{Dirichlet operator}, is bounded. Moreover, if $D$ is invertible, $\cd_0$ is related to $\cd_\lambda$ by
\begin{equation}\label{eq:8.1}
\cd_\lambda=(I-\lambda R(\lambda,D))\cd_0.
\end{equation}
\end{lemma}

We first discuss \eqref{eq:AIBVP} with $B=0$, and introduce the operator matrix
\[
\ca_0:=\begin{pmatrix}
A & 0 \\
0 & D_m
\end{pmatrix},\qquad D(\ca_0):=\left\{\begin{pmatrix}
x \\ u
\end{pmatrix}\in D(A)\times D(D_m) : \cp u=x\right\},
\]
on the product space $\cx$.

\begin{lemma}\label{lem:8.3}
Let $\lambda\in\rho(D)$. Then the domain $D(\ca_0)$ defined as above coincides with
\[
\ch:=\left\{\begin{pmatrix}
x \\ u
\end{pmatrix}\in D(A)\times X : -\cd_\lambda x+u\in D(D)\right\},
\]
and the identity
\begin{equation}\label{eq:8.2}
\lambda-\ca_0=\begin{pmatrix}
\lambda-A & 0 \\
0 & \lambda-D
\end{pmatrix}
\begin{pmatrix}
I_{\partial X} & 0 \\
-\cd_\lambda & I_X
\end{pmatrix}=:\cq_\lambda\crl
\end{equation}
holds.
\end{lemma}

\begin{proof}
To show the first claim, let us take a vector $\begin{pmatrix} x \\ u \end{pmatrix} \in\cx$. Then $\begin{pmatrix} x \\ u \end{pmatrix}\in\ch$ if and only if $x\in D(A)$ and $u-\cd_\lambda x\in\ker(\cp)$. But $\cd_\lambda x\in\ker(\lambda-D_m)\subseteq D(D_m)$, and hence $u-\cd_\lambda x\in D(D)=\ker(\cp)$ if and only if $u\in D(D_m)$ and $\cp(u-\cd_\lambda x)=\cp u-x=0$, i.e., $\cp u=x$. This proves that $D(\ca_0)=\ch$. It is clear that $\ch=D(\cq_\lambda\crl)$ since $\begin{pmatrix} x \\ u \end{pmatrix} \in D(\cq_\lambda\crl)$ if and only if
\[
\crl\begin{pmatrix} x \\ u \end{pmatrix}=\begin{pmatrix} x \\ -\cd_\lambda x+u \end{pmatrix}\in D(\cq_\lambda)=D(A)\times D(D).
\]
To show \eqref{eq:8.2} we take $\begin{pmatrix} x \\ u \end{pmatrix}\in D(\cq_\lambda\crl)$ and obtain
\[
\cq_\lambda\crl\begin{pmatrix} x \\ u \end{pmatrix}=\begin{pmatrix}
\lambda-A & 0 \\
0 & \lambda-D_m
\end{pmatrix}\begin{pmatrix}
x \\ -\cd_\lambda x+u
\end{pmatrix}=\begin{pmatrix}
(\lambda-A)x \\ (\lambda-D_m)u
\end{pmatrix}=(\lambda-\ca_0)\begin{pmatrix} x \\ u \end{pmatrix}.
\]
\end{proof}

Note that Lemma~\ref{lem:8.3} simply says that $\lambda-\ca_0$ is an osc operator matrix for $\lambda\in\rho(D)$. In particular, $\ca_0$ is osc whenever $D$ is invertible.

We now study the generator property for $\ca_0$.

\begin{theorem}\label{thm:8.4}
Assume that $D$ is invertible. Then $\ca_0$ generates a strongly continuous semigroup $(\ct_0(t))_{t\ge0}$ on $\cx$ if and only if for all $t\ge0$ the operator
\[
\tilde{Q}(t):D(A^2)\rightarrow X,\qquad \tilde{Q}(t)x:=-D\int_0^t S(t-s)\cd_0 T(s)x\;ds
\]
extends to a bounded linear operator $Q(t):\partial X\to X$, and this extension satisfies
\[
\limsup\limits_{t\rightarrow0^+}\norm{Q(t)}_{\mathcal{L}(\partial X,X)}<\infty.
\]
In this case, $(\ct_0(t))_{t\ge0}$ is given as in \eqref{eq:6.2} by
\[
\ct_0(t)=\begin{pmatrix}
T(t) & 0 \\
Q(t) & S(t)
\end{pmatrix},\qquad t\geq0.
\]
\end{theorem}

\begin{proof}
By Lemma~\ref{lem:8.3}, the claim is only a consequence of Theorem~\ref{thm:6.1}.
\end{proof}

\begin{corollary}\label{cor:8.5}
Assume that $A$ and $D$ generate analytic semigroups on $\partial X$ and $X$, respectively. Then $\ca_0$ generates an analytic semigroup on $\cx$.
\end{corollary}

The proof is based on similarity and perturbation arguments, and can be found in \cite[Cor.~2.8]{CENN02}.

We now turn to the original version of \eqref{eq:AIBVP} and look at the operator matrix
\[
\ca:=\ca_0+\cb,
\]
where
\[
\cb:=\begin{pmatrix}
0 & B \\
0 & 0
\end{pmatrix}, \qquad D(\cb):=\dx\times D(B),
\]
with $D(\ca)=D(\ca_0)$ since $D(B)\supseteq D(D_m)$. Again, we can show that the matrix $\lambda-\ca$ is osc.

\begin{lemma}\label{lem:8.6}
Let $\lambda\in\rho(D)$. Then $\lambda-\ca$ is an osc operator matrix. More precisely, $\lambda-\ca$ can be factorized as
\begin{equation}\label{eq:8.3}
\lambda-\ca = \begin{pmatrix}
\lambda-A-B\cd_\lambda & -B \\
0 & \lambda-D
\end{pmatrix}
\begin{pmatrix}
I_\dx & 0 \\
-\cd_\lambda & I_X
\end{pmatrix}.
\end{equation}
\end{lemma}

\begin{proof}
Taking into account the identity \eqref{eq:8.2}, a direct computation yields
\[
\begin{aligned}
\lambda-\ca &= \cq_\lambda\crl-\cb = (\cq_\lambda- \cb\crl^{-1})\crl \\
&= \bigg[\begin{pmatrix} \lambda-A & 0 \\ 0 & \lambda-D \end{pmatrix}-\begin{pmatrix} B\cd_\lambda & B \\ 0 & 0 \end{pmatrix}\bigg]\crl \\
&= \begin{pmatrix} \lambda-A-B\cd_\lambda & -B \\ 0 & \lambda-D \end{pmatrix}\begin{pmatrix} I_{\partial X} & 0 \\ -\cd_\lambda & I_X \end{pmatrix}
\end{aligned}
\]
with correct domains since $D(\ca)=D(\ca_0)$.
\end{proof}

If $B$ is a bounded operator, the generator property of $\ca$ is equivalent to the generator property of $\ca_0$ which has been characterized in Theorem~\ref{thm:8.4}. In the case of unbounded feedback $B$, Proposition~\ref{prop:6.3} can be used.

\begin{proposition}\label{prop:8.7}
Let $A$ be bounded, and $B$ be relatively $\begin{pmatrix} \cp \\ D_m \end{pmatrix}$-bounded. Then the operator $\ca$ generates a strongly continuous (resp., analytic) semigroup on $\cx$ if and only if for some $\lambda_0\in\rho(D)$ the operator $D-\cd_{\lambda_0} B$, defined on $D(D)$, generates a strongly continuous (resp., analytic) semigroup on $X$.
\end{proposition}

\begin{proof}
Observe first that the relative $\begin{pmatrix} \cp \\ D_m \end{pmatrix}$-boundedness of $B$ implies that $B\in\mathcal{L}([D(D)],\dx)$. By \eqref{eq:8.3} we obtain
\[
\begin{aligned}
\ca &= \left[\begin{pmatrix} A+B\cd_{\lambda_0} & B \\ 0 & D \end{pmatrix}\begin{pmatrix} I_\dx & 0 \\ -D_{\lambda_0} & I_X \end{pmatrix}-\lambda_0\begin{pmatrix} I_\dx & 0 \\ -\cd_{\lambda_0} & I_X \end{pmatrix}\right] + \lambda_0\begin{pmatrix} I_\dx & 0 \\ 0 & I_X \end{pmatrix} \\
&= \begin{pmatrix} A+B\cd_{\lambda_0} & B \\ 0 & D \end{pmatrix}\begin{pmatrix} I_{\partial X} & 0 \\ -\cd_{\lambda_0} & I_X \end{pmatrix} + \lambda_0\begin{pmatrix} 0 & 0 \\ \cd_{\lambda_0} & 0 \end{pmatrix}=: \tilde{\ca}_{\lambda_0}{\mathcal R}_{\lambda_0}+\lambda_0\cp_{\lambda_0}.
\end{aligned}
\]
By Lemma~\ref{lem:8.2}, the operator $\lambda_0\cp_{\lambda_0}$ is a bounded perturbation of the osc matrix $\tilde{\ca}_{\lambda_0}{\mathcal R}_{\lambda_0}$. If we can show that $B\cd_{\lambda_0}$ is bounded, then the claim follows by Proposition~\ref{prop:6.3}. The closedness of $\begin{pmatrix} \cp \\ D_m \end{pmatrix}$ implies the closedness of the operator matrix
\[
\cl:=\begin{pmatrix} 0 & \cp \\ 0 & D_m \end{pmatrix},\qquad D(\mathcal{L}):=\dx\times D(D_m),
\]
on $\cx$, and hence $[D({\mathcal L})]$ is a Banach space continuously embedded in $\cx$. Observe now that, by Lemma~\ref{lem:8.2}, the operator
\[
{\mathcal H}_{\lambda_0}:=\begin{pmatrix} 0 & 0 \\ \cd_{\lambda_0} & 0 \end{pmatrix}: \cx\rightarrow\cx
\]
is bounded from $\cx$ to $\cx$ and its range is contained in $[D({\mathcal L})]$. It then follows by \cite[Cor.~B.7]{EN00} that ${\mathcal H}_{\lambda_0}$ is also bounded from $\cx$ to $[D({\mathcal L})]$, and since under the above assumptions $\cb$ is bounded from $[D({\mathcal L})]$ to $\cx$,
\[
\cb{\mathcal H}_{\lambda_0}=\begin{pmatrix} 0 & 0 \\ 0 & B{\mathcal D}_{\lambda_0} \end{pmatrix}: \cx \rightarrow \cx
\]
is bounded, and therefore $A+B\cd_{\lambda_0}$ is bounded.
\end{proof}

Again by Lemma~\ref{lem:8.6} we obtain spectral properties of $\ca$ analogous to those stated in Theorem~\ref{thm:5.1}.

\begin{proposition}\label{prop:8.8}
Under the Assumptions~\ref{assumptions:8.1} for every $\lambda\in\rho(D)$ one has the equivalence
\[
\lambda\in\rho(\ca) \iff \lambda\in\rho(A+B\cd_\lambda).
\]
\end{proposition}

\begin{proof}
Of course, $\lambda\in\rho(\ca)$ if and only if $0\in\rho(\lambda-\ca)$. Since, by Lemma~\ref{lem:8.6}, $\lambda-\ca$ is osc, we now apply Theorem~\ref{thm:5.1}, which reads by \eqref{eq:8.3} as
\[
0\in\rho(\lambda-\ca)\iff 0\in\rho(\lambda-A-B\cd_\lambda).
\]
\end{proof}

\begin{remark}
\label{rem:8.9}
In the context of abstract initial-boundary value problems Remark~\ref{rem:5.2} reads by \eqref{eq:8.3} as follows: If the boundary space $\dx$ is finite dimensional, then $\lambda\in\rho(D)$ lies in the spectrum of $\ca$ if and only if it is an eigenvalue of $A+B\cd_\lambda$.
\end{remark}

We now look at the positivity of the semigroup $\ttt$ generated by $\ca$ and reformulate Theorem~\ref{thm:7.1} and Corollary~\ref{cor:7.3}.

\begin{proposition}\label{prop:8.10}
Let $\partial X$ and $X$ be Banach lattices and let $\ca$ generate a strongly continuous semigroup $\ttt$ on $\cx = \dx\times X$. Assume further that $A$ and $D$ generate \textit{positive} strongly continuous semigroups on $\partial X$ and $X$, respectively. Then the following hold.
\begin{enumerate}[label=(\roman*)]
\item If the feedback operator $B$ is bounded and positive and the Dirichlet operators $\cd_\lambda$ are positive for $\lambda$ sufficiently large, then $\ttt$ is a positive semigroup.
\item Let $\ttt$ be positive. Then it is exponentially stable if (and only if) the spectral bounds $s(A+B\cd_0)$ and $s(D)$ are both strictly negative.
\end{enumerate}
\end{proposition}

\begin{proof}
(i) Observe first that in the particular case of the unperturbed matrix $\ca_0$ the formula \eqref{eq:5.1} for the resolvent becomes
\[
R(\lambda, \ca_0) = \begin{pmatrix}
R(\lambda, A) & 0 \\
\cd_{\lambda} R(\lambda,A) & R(\lambda, D)
\end{pmatrix}.
\]
Hence if $A$ and $D$ are resolvent positive, and if $\cd_\lambda\ge0$ for $\lambda$ large, then $\ca_0$ is resolvent positive and generates a positive semigroup. Since $B$ is bounded and positive, the same holds for $\cb$, and hence the semigroup $\ttt$ generated by $\ca=\ca_0+\cb$ is given by the Dyson–Phillips expansion (cf. \cite[Thm.~III.1.10]{EN00}) that consists of positive terms only. Hence $\ttt$ is positive as claimed.

(ii) We first note that $D$ is invertible, and hence, by Lemma~\ref{lem:8.6}, $\ca$ is an osc operator matrix. The claim is then just a reformulation of Corollary~\ref{cor:7.3}. For a proof of the ``only if'' part we refer to \cite[Thm.~5.4]{CENN02}.
\end{proof}

We now discuss a problem fitting into the above framework, cf. \cite[\S~3 and \S~6]{CENN02}.

\begin{example}\label{ex:8.11}
We consider a heat equation on a bounded domain $\Omega\subseteq \mathbb{R}^n$ with dynamical boundary conditions and assume that we can control the dynamics by a boundary feedback. More precisely, we consider the system
\[
\begin{cases}
\dot{u}(t,z) = \Delta_{\partial\Omega}u(t,z)-\mu(z)\,u(t,z)+\phi(t,z), & t\geq0,\; z\in\partial\Omega, \\
\dot{u}(t,x) = \Delta_{\Omega}u(t,x), & t\geq0,\; x\in\Omega, \\
u(0,z) = f(z), & z\in\partial\Omega, \\
u(0,x) = g(x), & x\in\Omega.
\end{cases}
\tag{HEBF}\label{eq:HEBF}
\]
Here, $\phi:\pos\times\doo\rightarrow\mathbb{R}$ is defined as
\[
\phi(t,z):=\int_\Omega k(z,x)u(t,x)dx\qquad\text{for all } t\ge0 \text{ and } z\in\partial\Omega,
\]
with kernel $k\in L^2(\partial\Omega\times\Omega)$.

We assume that $0<\epsilon\leq \mu\in L^\infty(\doo)$ and take the initial data $u_0$ and $u_1$ in $L^2(\partial\Omega)$ and $L^2(\partial\Omega)$, respectively. We also denote by $\Delta_{\partial\Omega}$ the Laplace–Beltrami operator on the boundary $\partial\Omega$, which is assumed to be a compact manifold without boundary.

In order to show that the above evolution equation is governed by a strongly continuous semigroup, we reformulate \eqref{eq:HEBF} as \eqref{eq:AIBVP} with initial data $(x_0,u_0)=(f,g)$ by considering the state space $X:=L^2(\Omega)$, the boundary space $\partial X:=L^2(\partial\Omega)$, and the following operators:
\begin{itemize}
\item $Av:=\Delta_{\partial\Omega}v-\mu(\cdot)v$ with domain $D(A):=\left\{v\in H^{1\over2}(\partial\Omega)\cap H^2\loc(\partial\Omega): \Delta_{\partial\Omega}v\in L^2(\partial\Omega)\right\}$,
\item $Bv:=\int_\Omega k(x,\cdot)v(x)dx$ mapping $L^2(\Omega)$ in $L^2(\partial\Omega)$,
\item $D_m v:=\Delta_{\Omega}v$ with domain $D(D_m):=\left\{v\in H^{1\over2}(\Omega)\cap H^2\loc(\Omega) : \Delta_{\Omega}v\in L^2(\Omega)\right\}$,
\item $\cp v:=v|_{L^2(\partial\Omega)}$ for $v\in D(D_m)$ in the sense of traces, see \cite[\S~2.7]{LM72}.
\end{itemize}
\end{example}

\begin{remark}\label{rem:8.12}
In the definition of $\cp$, and also in the sequel, we consider traces of $L^2$ functions, cf. \cite[Chapter~2]{LM72}. In particular, $D_m$ is defined on the maximal domain such that the traces of its elements exist as $L^2(\partial\Omega)$ functions (see \cite[Thm.~2.7.4]{LM72}). This is also an explanation for the notion of \textit{maximal operator}.
\end{remark}

In order to apply the above generation and stability results we have to check that the Assumptions~\ref{assumptions:8.1} are fulfilled. We again refer to \cite{CENN02} for details about surjectivity of $\cp$ and closedness of $\begin{pmatrix} \cp \\ D_m \end{pmatrix}$, and only point out that these proofs require non-trivial results about interpolation spaces (as a reference, see \cite[Chapters~1--2]{LM72}) and interior estimates for the Laplacian (see, e.g., \cite[Chapter~8]{GT77}. Our goal is now to show that Corollary~\ref{cor:8.5} applies. The Laplace–Beltrami operator $\Delta_{\partial\Omega}$ on the manifold $\partial\Omega$ generates an analytic semigroup, cf. \cite[Thm.~13.7.1]{Ta96}, and due to the boundedness of the multiplication operator $M_\mu u:=\mu u$ on $\dx$ so does $A$. While $(D_m,D(D_m))$ is \textit{not} a generator, its restriction $D$ is the Dirichlet-Laplacian on $L^2(\Omega)$, and hence generates an analytic semigroup (see, e.g., \cite[Thm.~7.2.7]{Pa83}). By Corollary~\ref{cor:8.5}, the osc operator matrix $\ca_0$ is the generator of an analytic semigroup on $\cx$. Moreover, $B$ is a bounded operator, and hence $\cb=\left(\begin{smallmatrix} 0 & B \\ 0 & 0 \end{smallmatrix}\right)$ is a bounded perturbation. Therefore, also the operator matrix $\ca=\ca_0+\cb$ is the generator of an analytic semigroup. Both the Laplace–Beltrami operator on $\partial\Omega$ and the perturbed Laplace operator on $\Omega$ are resolvent positive and hence generate positive semigroups on $L^2$. Since it can be shown, using the weak maximum principle (see \cite[Thm.~8.1]{GT77}), that all Dirichlet operators $\cd_\lambda$, $\lambda>0$, are positive, it follows from Proposition~\ref{prop:8.10}(i) that $\ttt$ is positive whenever the integral kernel $k$ is a.e. positive, and hence $\cb\geq0$. Finally, under the above positivity assumptions on $k$ and $\mu$, the Dirichlet Laplacian has strictly negative spectral bound. Therefore, as a consequence of Proposition~\ref{prop:8.10}(ii), we obtain that $\ttt$ is exponentially stable if and only if $s(A+B\cd_0)<0$. Roughly speaking, this means that the asymptotics of the semigroup governing our evolution equation \eqref{eq:HEBF} is not influenced by the internal dynamics of the problem. Summing up, we can say that
\begin{quote}
\emph{stability of the problem is independent of the (internal) diffusion.}
\end{quote}

\section{A DIFFUSION–TRANSPORT SYSTEM WITH DYNAMICAL BOUNDARY CONDITIONS}
\label{sec:diffusion-transport}

In this section, we consider a specific problem and show how our abstract tools lead to very concrete and complete informations. To this purpose, we consider the equations
\[
\begin{cases}
\dot{u}(t,x) = u''(t,x)+ku'(t,x), & t\geq0,\; x\in[0,1], \\
\dot{u}(t,0) = u'(t,0)+\alpha u(t,0), & t\geq0, \\
\dot{u}(t,1) = -u'(t,1)+\beta u(t,1), & t\geq0, \\
u(0,x) = f(x), & x\in[0,1], \\
u(0,0) = u_0,\quad u(0,1)=u_1,
\end{cases}
\tag{DTDB}\label{eq:DTDB}
\]
where $\alpha,\beta,u_0,u_1\in\mathbb{C}$, $k\geq0$, $f:[0,1]\to\mathbb{C}$.

Such a system has already been considered in \cite{KMN02}, where we discussed its well-posedness in $L^2(0,1)$, while results in $C[0,1]$ have been obtained, e.g., in \cite{FGGR00,Wa03}. An $n$-dimensional version of this system was introduced in \cite{AE96}, and has also been studied in \cite{FGGR02} by using weighted spaces and energy estimates and in \cite{AMPR02} by means of variational methods.

The system \eqref{eq:DTDB} becomes an abstract initial boundary value problem \eqref{eq:AIBVP} as studied in the previous section if we consider the state space $X:=L^2(0,1)$, the boundary space $\dx:=\mathbb{C}^2$, the operator
\[
A:=\begin{pmatrix}
\alpha & 0 \\
0 & \beta
\end{pmatrix}
\]
on $\mathbb{C}^2$, the maximal operator
\[
D_m u:=u''+ku',\qquad D(D_m):=H^2(0,1),
\]
and its restriction $D$ defined on
\[
D(D):=H^2(0,1)\cap H^1_0(0,1),
\]
the boundary operator
\[
\cp u:=\begin{pmatrix} u(0) \\ u(1) \end{pmatrix},\qquad D(\cp):=D(D_m),
\]
and the feedback operator
\[
Bu:=\begin{pmatrix} u'(0) \\ -u'(1) \end{pmatrix},\qquad D(B):=D(D_m).
\]

We are now in the position to discuss the well-posedness of \eqref{eq:DTDB}. As explained in the previous section, \eqref{eq:DTDB} can be reformulated as an abstract Cauchy problem \eqref{eq:ACP} on the product space $\cx:=\dx\times X$, where
\[
\cu(t):=\begin{pmatrix} \cp u(t) \\ u(t) \end{pmatrix},\; t\geq0,\qquad \cu_0:=\begin{pmatrix} \begin{pmatrix} u_0 \\ u_1 \end{pmatrix} \\ f \end{pmatrix},
\]
and the matrix $\ca$ is given by
\[
\ca:=\begin{pmatrix}
A & B \\
0 & D_m
\end{pmatrix},\qquad D(\ca):=\left\{ \begin{pmatrix} v \\ u \end{pmatrix}\in \dx\times D(D_m):\cp u=v \right\}.
\]

\begin{lemma}\label{lem:9.1}
The Assumptions~\ref{assumptions:8.1} hold.
\end{lemma}

\begin{proof}
The boundedness of $A$ yields (B$_1$). The operator $\frac{d^2}{dx^2}$ with Dirichlet boundary conditions, hence the operator $D$ generates an analytic semigroup and (B$_2$) is satisfied. Clearly, $\cp$ is surjective onto $\mathbb{C}^2$, that is, (B$_3$) holds. Finally, one can show as in \cite[Lemma~3.3]{CENN02} that the operator
\[
\cl_1:=\begin{pmatrix} \cp \\ \frac{d^2}{dx^2} \end{pmatrix},\qquad D(\cl_1):=D(D_m),
\]
is closed. Since
\[
\cl_2:=\begin{pmatrix} 0 \\ k\frac{d}{dx} \end{pmatrix},\qquad D(\cl_2):=D(D_m),
\]
is relatively $\cl_1$-bounded with $\cl_1$-bound 0 for $k\geq0$, their sum $\cl_1+\cl_2=\begin{pmatrix} \cp \\ D_m \end{pmatrix}$ is also closed (see \cite[Lemma~III.2.4]{EN00}).
\end{proof}

Therefore, Lemma~\ref{lem:8.2} applies and we can consider the bounded Dirichlet operators $\cd_\lambda$, $\lambda\in\rho(D)$, mapping each pair $\begin{pmatrix} x_0 \\ x_1 \end{pmatrix}\in\mathbb{C}^2$ onto the unique solution of the ordinary differential equation
\begin{equation}\label{eq:9.1}
\begin{cases}
u''(x)+ku'(x)-\lambda u(x)=0, & x\in(0,1), \\
u(0)=x_0,\qquad u(1)=x_1.
\end{cases}
\end{equation}

\begin{lemma}\label{lem:9.2}
The operator $B$ is relatively $D_m$-bounded.
\end{lemma}

\begin{proof}
The first derivative on $L^2(0,1)$ is relatively bounded by the second derivative, with relative bound 0, and hence the graph norm of the second derivative and the graph norm of $D_m$ are equivalent. It then follows from the embedding $H^1(0,1)\hookrightarrow C[0,1]$ that we can find suitable constants $\xi$, $\xi_1$, $\xi_2$, $\tilde\xi$, such that
\[
\begin{aligned}
\norm{Bu} &= \left(|u'(0)|^2 + |u'(1)|^2\right)^{1\over 2} \leq |u'(0)| + |u'(1)| \leq 2\norm{u'}_{C[0,1]} \\
&\leq \xi\left(\norm{u''}_{L^2(0,1)}+ \norm{u'}_{L^2(0,1)}\right) \leq \xi_1\norm{u''}_{L^2(0,1)}+\xi_2\norm{u}_{L^2(0,1)} \\
&\leq \tilde\xi\norm{u}_{D_m},
\end{aligned}
\]
for all $u\in D(D_m)$.
\end{proof}

\begin{lemma}\label{lem:9.3}
The operator $D$ has compact resolvent and is invertible.
\end{lemma}

\begin{proof}
Since $D$ has nonempty resolvent set and is defined on a subspace of $H^1(0,1)$, the claim follows by standard embedding theorems (see \cite[Ex.~II.4.30(4)]{EN00}). Hence, $P\sigma(D)=\sigma(D)$, i.e., its spectral values can be determined by solving the homogeneous ordinary differential equation \eqref{eq:9.1} for $x_0=x_1=0$. This yields
\[
\sigma(D)=\left\{-\pi^2 n^2-\frac{k^2}{4} : n=1,2,\ldots\right\}.
\]
In particular, $s(D)<0$ and $D$ is invertible.
\end{proof}

Lemma~\ref{lem:8.2} and Lemma~\ref{lem:9.2} show that the assumption (A$_3$) is satisfied, and hence, by Lemma~\ref{lem:8.3}, the operator $\ca$ is an osc matrix that, by \eqref{eq:8.3}, can be factorized as
\begin{equation}\label{eq:9.2}
\ca = \begin{pmatrix}
A+B\cd_0 & B \\
0 & D
\end{pmatrix}
\begin{pmatrix}
I_{\mathbb{C}^2} & 0 \\
-\cd_0 & I_{L^2(0,1)}
\end{pmatrix}.
\end{equation}

\begin{theorem}
\label{thm:9.4}
The operator matrix $\ca$ generates an analytic semigroup $\ttt$ on $\cx$.
\end{theorem}

\begin{proof}
By Proposition~\ref{prop:8.7} it suffices to show that $D-\cd_0B$, with domain $D(D)$, generates an analytic semigroup on $X$. This follows since $D$ generates an analytic semigroup and since $\cd_0 B$ is relatively $D$-compact (use \cite[Cor.~III.2.17]{EN00}).
\end{proof}

We now investigate further properties of the semigroup that governs the system \eqref{eq:DTDB}.

\begin{proposition}\label{prop:9.5}
The semigroup $\ttt$ is immediately compact.
\end{proposition}

\begin{proof}
Recall that an analytic semigroup is immediately norm continuous. It then suffices to show that $\ca$ has compact resolvent (see \cite[Thm.~II.4.29]{EN00}). Let us assume, without loss of generality, that the bounded operator $\tilde A:= A+B\cd_0$ is invertible. Then, by \eqref{eq:5.1}, we obtain the matrix representation for the resolvent in $\lambda=0$ as
\[
R(0,\ca)=\begin{pmatrix}
-\tilde A^{-1} & \tilde A^{-1}B D^{-1} \\
-\cd_0\tilde A^{-1} & -D^{-1} + \cd_0 \tilde A^{-1} B D^{-1}
\end{pmatrix}.
\]
By Lemma~\ref{lem:9.3}, $D^{-1}$ is compact. Moreover, $B$ is relatively $D$-bounded and $D^{-1}$ is bounded, hence $BD^{-1}:X\rightarrow \dx$ is compact, since it is bounded and has finite-dimensional range. It follows that $\cd_0 \tilde A^{-1} BD^{-1}$, and hence $-D^{-1}+\cd_0 \tilde A^{-1} B D^{-1}$ is compact. The other entries of $R(0,\ca)$ are bounded and have finite-dimensional range.
\end{proof}

Taking into account Proposition~\ref{prop:8.8} we can now investigate the spectrum of $\ca$. Since the spectrum of $D$ has already been computed in the proof of Lemma~\ref{lem:9.3}, by Remark~\ref{rem:8.9} it suffices to study the spectrum of $A+B\cd_\lambda$. Solving the equation \eqref{eq:9.1} for $\lambda\not=-\frac{k^2}{4}$ we see that the Dirichlet operator $\cd_\lambda$ is given by
\[
\cd_\lambda\begin{pmatrix} x_0 \\ x_1 \end{pmatrix}(x)=\left( \frac{x_1-x_0 e^{\md}}{e^{\mo}-e^{\md}}\right)e^{\mo x}+ \left(\frac{x_0 e^{\mo}-x_1}{e^{\mo}-e^{\md}}\right)e^{\md x}, \qquad x\in[0,1],
\]
for all $x_0,x_1\in\mathbb{C}$, and hence
\begin{equation}\label{eq:9.3}
B\cd_\lambda=\frac{1}{e^\mo-e^\md}\begin{pmatrix}
\md e^\mo-\mo e^\md & \mo-\md \\
(\mo-\md)e^{\mo+\md} & \md e^\md - \mo e^\mo
\end{pmatrix},
\end{equation}
where $\mo,\md$ are the roots of the characteristic polynomial of \eqref{eq:9.1}, i.e.,
\[
\mo=\frac{-k+\sqrt{k^2+4\lambda}}{2},\qquad \md=\frac{-k-\sqrt{k^2+4\lambda}}{2}.
\]
If $\lambda=-\frac{k^2}{4}$, we have
\[
\cd_{-\frac{k^2}{4}}\begin{pmatrix} x_0 \\ x_1 \end{pmatrix}(x)=x_0 e^{-\frac{kx}{2}} + \left(x_1 e^{\frac{k}{2}} - x_0\right)x e^{-\frac{kx}{2}}, \qquad x\in[0,1],
\]
for all $x_0,x_1\in\mathbb{C}$, and
\begin{equation}\label{eq:9.4}
B\cd_{-\frac{k^2}{4}}=\begin{pmatrix}
-1-\frac{k}{2} & e^{\frac{k}{2}} \\
e^{-\frac{k}{2}} & -1+\frac{k}{2}
\end{pmatrix}.
\end{equation}

By Remark~\ref{rem:8.9}, we now obtain a characteristic equation for the spectrum of $\ca$:
\begin{itemize}
\item if $\lambda\not=-\pi^2 n^2-\frac{k^2}{4}$, $n=0,1,2,\ldots$, then $\lambda\in\sigma(\ca)$ if and only if
\[
\lambda^2- \lambda\left(\alpha+\beta-1-\frac{(\mo -\md)(e^\mo + e^\md)}{e^\mo-e^\md}\right)+\alpha\beta+\alpha\frac{\md e^\md - \mo e^\mo}{e^\mo-e^\md}+\beta\frac{\md e^\mo-\mo e^\md}{e^\mo-e^\md}=0;
\]
\item $\lambda=-\frac{k^2}{4}\in\sigma(\ca)$ if and only if
\[
k^4+4k^2(\alpha+\beta-3)+8k(\alpha-\beta)+16(\alpha\beta- \alpha-\beta)=0.
\]
\end{itemize}

Using the above computations, we can also prove the following.

\begin{proposition}\label{prop:9.6}
If the coefficients $\alpha,\beta$ are real, then the semigroup $\ttt$ is positive.
\end{proposition}

\begin{proof}
We check the assumptions (b.i,ii,iii') of Theorem~\ref{thm:7.1} and Remark~\ref{rem:7.2}. By \eqref{eq:9.2}, $A_\lambda$ is given by
\[
A_\lambda=A+B\cd_0-\lambda BR(\lambda,D)\cd_0= A+B(I-\lambda R(\lambda,D))\cd_0=A+B\cd_\lambda
\]
for all $\lambda>0$, where the last identity follows by \eqref{eq:8.1}. Using \eqref{eq:9.3}, this yields an explicit representation of $A_\lambda$ as a $2\times 2$ scalar matrix.

Let $\lambda>0$, and observe that $\mo$ and $\md$ are real, and $\md\le-\mo\le0$ holds. It follows that $A_\lambda$ is real and positive off-diagonal, hence it generates a positive semigroup and, therefore, the resolvent $R(\;\cdot\;,A_\lambda)$ is positive on the halfline $(s(A_\lambda),+\infty)$, cf. \cite[Lemma~VI.1.9]{EN00}. To show that, in particular, $R(\lambda,A_\lambda)\geq0$, we prove that $\norm{A_\lambda}$, and hence $s(A_\lambda)$ is strictly less than $\lambda$, for $\lambda$ large enough. Indeed, the diagonal entries of $A_\lambda$ are dominated by
\[
\max\{\alpha,\beta\}+\mo = \max\{\alpha,\beta\}-\frac{k}{2}+ \frac{\sqrt{k^2+4\lambda}}{2}<\lambda,
\]
for $\lambda$ large enough. Further, the lower left entry in \eqref{eq:9.3} is less than the upper right $(\mo-\md)(e^\mo-e^\md)^{-1}$ which tends to $0$ as $\lambda\to+\infty$, and we conclude that all entries of $A_\lambda$ are strictly less than $\lambda$ for $\lambda$ large enough.

Elementary calculus arguments show that $B$ is positive on $D(D)\subset C^1[0,1]\cap C_0[0,1]$. Finally, consider the operator $L_\lambda=DR(\lambda,D)\cd_0$. In order to show that $-L_\lambda$ is positive for $\lambda$ large enough it suffices to show the positivity of its factors $-D,R(\lambda,D),\cd_0$. This will complete the proof.

Observe that $D_m$ satisfies the weak maximum principle, and accordingly the solution $u$ of
\[
\begin{cases}
-&u''(x)-ku'(x)=f(x),\qquad x\in(0,1), \\
&u(0)=x_0,\qquad u(1)=x_1,
\end{cases}
\]
is positive if $f\equiv 0$ and $\begin{pmatrix} x_0 \\ x_1 \end{pmatrix}\geq0$, hence $\cd_0$ is positive. Taking instead $f\geq0$ and $x_0=x_1=0$, by the same argument $-D$ is positive. To conclude, assume that $u\geq0$ and set $v:=R(\lambda,D)u$ for $\lambda>0$ large enough. Then $v$ satisfies
\[
\begin{cases}
-&v''(x)-kv'(x)+\lambda v(x)=u(x)\geq 0,\qquad x\in(0,1), \\
&v(0)=0,\qquad v(1)=0,
\end{cases}
\]
and again the weak maximum principle yields $v(x)\geq0$, i.e., $D$ is resolvent positive.
\end{proof}

In the rest of this section we will often need the following computation, which allows to characterize further properties of the semigroup generated by $\ca$. Take ${\mathcal U}:=\begin{pmatrix} \begin{pmatrix} u(0) \\ u(1) \end{pmatrix} \\ u \end{pmatrix},{\mathcal V}:=\begin{pmatrix} \begin{pmatrix} v(0) \\ v(1) \end{pmatrix} \\ v \end{pmatrix}\in D(\ca)$. We then obtain
\begin{equation}
\label{eq:9.5}
\begin{aligned}
\left<\ca\cu,{\mathcal V}\right> &= \left<\ca\begin{pmatrix} \begin{pmatrix} u(0) \\ u(1) \end{pmatrix} \\ u \end{pmatrix},\begin{pmatrix} \begin{pmatrix} v(0) \\ v(1) \end{pmatrix} \\ v \end{pmatrix}\right> \\
&= \left<\begin{pmatrix} \begin{pmatrix} u'(0)+\alpha u(0) \\ -u'(1)+\beta u(1) \end{pmatrix} \\ u''+ku' \end{pmatrix}, \begin{pmatrix} \begin{pmatrix} v(0) \\ v(1) \end{pmatrix} \\ v \end{pmatrix}\right> \\
&= \int_0^1 u''(x)\overline{v(x)}\;dx+k\int_0^1 u'(x)\overline{v(x)}\;dx \\
&\quad + u'(0)\overline{v(0)}+\alpha u(0) \overline{v(0)}-u'(1)\overline{v(1)}+\beta u(1)\overline{v(1)} \\
&= \Big[u'(x)\overline{v(x)}\Big]_0^1-\int_0^1 u'(x)\overline{v'(x)}\;dx+ k\Big[u(x)\overline{v(x)}\Big]_0^1- k\int_0^1 u(x)\overline{v'(x)}\;dx \\
&\quad + u'(0)\overline{v(0)}+\alpha u(0)\overline{v(0)}-u'(1)\overline{v(1)}+\beta u(1)\overline{v(1)} \\
&= -\int_0^1 u'(x)\overline{v'(x)}\;dx-k\int_0^1 u(x)\overline{v'(x)}\;dx+ k\Big[u(x)\overline{v(x)}\Big]_0^1 +\alpha u(0)\overline{v(0)}+\beta u(1)\overline{v(1)}.
\end{aligned}
\end{equation}

\begin{proposition}\label{prop:9.7}
If the coefficients $\alpha$, $\beta$, and $k$ satisfy
\begin{equation}\tag{D}\label{eq:D}
\operatorname{Re}\alpha\leq\frac{k}{2}\leq-\operatorname{Re}\beta,
\end{equation}
then the semigroup $\ttt$ is contractive.
\end{proposition}

\begin{proof}
By the Lumer–Phillips Theorem we only need to show that the operator $\ca$ is dissipative. Take $\cu:=\begin{pmatrix} \begin{pmatrix} u(0) \\ u(1) \end{pmatrix} \\ u \end{pmatrix}\in D(\ca)$. By \eqref{eq:9.5}
\[
\operatorname{Re}\left<\ca\cu,\cu\right>=-\int_0^1 |u'(x)|^2dx+ \left(\operatorname{Re}\beta +\frac{k}{2}\right)|u(1)|^2+\left(\operatorname{Re}\alpha -\frac{k}{2}\right)|u(0)|^2,
\]
which is negative under the assumption \eqref{eq:D}.
\end{proof}

\begin{proposition}\label{prop:9.8}
The operator $\ca$, hence the semigroup $\ttt$ is self-adjoint if and only if
\[
k=0\qquad\text{and}\qquad\alpha,\beta\in \mathbb{R}.
\tag{SA}\label{eq:SA}
\]
\end{proposition}

\begin{proof}
We show that $\ca$ is symmetric if and only if \eqref{eq:SA} holds. By \eqref{eq:9.5} we obtain that
\[
\left<\ca\cu,{\mathcal V}\right>=-\int_0^1 u'(x)\overline{v'(x)}\;dx-k\int_0^1 u(x)\overline{v'(x)}\;dx+ k\Big[u(x)\overline{v(x)}\Big]_0^1 +\alpha u(0)\overline{v(0)}+\beta u(1)\overline{v(1)}
\]
and
\[
\left<\cu,\ca{\mathcal V}\right>=-\int_0^1 u'(x)\overline{v'(x)}\;dx-k\int_0^1 u'(x)\overline{v(x)}\;dx+ k\Big[u(x)\overline{v(x)}\Big]_0^1 +\overline{\alpha} u(0)\overline{v(0)}+\overline{\beta} u(1)\overline{v(1)}
\]
for all $\cu:=\begin{pmatrix} \begin{pmatrix} u(0) \\ u(1) \end{pmatrix} \\ u \end{pmatrix},{\mathcal V}:=\begin{pmatrix} \begin{pmatrix} v(0) \\ v(1) \end{pmatrix} \\ v \end{pmatrix}\in D(\ca)$. These values are equal for all $\cu,{\mathcal V}$ if and only if \eqref{eq:SA} holds. Finally, since $\ca$ is a generator, symmetry implies self-adjointness.
\end{proof}

\begin{remark}\label{rem:9.9}
Observe that under the assumptions \eqref{eq:D} and \eqref{eq:SA}, i.e., if $k=0$ and the constants $\alpha,\beta$ are real and negative, then the semigroup $\ttt$ is bounded analytic of angle $\frac{\pi}{2}$.
\end{remark}

\begin{corollary}
\label{cor:9.10}
The semigroup $\ttt$ is symmetric submarkovian (markovian, resp.) if and only if $k=0$ and $\alpha,\beta$ are real and negative ($k=\alpha=\beta=0$, resp.).
\end{corollary}

\begin{proof}
Use Proposition~\ref{prop:9.6} and Proposition~\ref{prop:9.8} and observe that
\[
\ca\mathbf{1}=\begin{pmatrix} \begin{pmatrix} \alpha \\ \beta \end{pmatrix} \\ 0 \end{pmatrix}.
\]
\end{proof}

\begin{remark}\label{rem:9.11}
Let $\alpha,\beta$ be real and negative. Then $\ca$ is the generator of positive analytic submarkovian semigroup, and arguing as in \cite[Thm.~X.55]{RS75} one can show that $\ttt$ is continuous $L^p$-contractive in the sense of \cite[\S~X.8, p.~255]{RS75}, and also analytic on $\mathbb{C}^2\times L^p(0,1)$, $1<p<\infty$.
\end{remark}

Frequently, when our interest turns to stability properties of the solutions of \eqref{eq:DTDB}, positivity enters the picture and permits quite simple and explicit criteria.

\begin{proposition}\label{prop:9.12}
Let the coefficients $\alpha,\beta$ be real.
\begin{enumerate}[label=(\roman*)]
\item In the case of $k=0$, the semigroup $\ttt$ is uniformly exponentially stable if and only if
\[
\alpha+\beta<\min\{2,\alpha\beta\}.
\tag{UES$_{k=0}$}\label{eq:UES-k0}
\]
\item In the case of $k>0$, the semigroup $\ttt$ is uniformly exponentially stable if and only if
\[
\frac{\alpha +\beta}{1+e^{-k}}< \frac{k}{1-e^{-k}} < \frac{\alpha\beta}{\beta +\alpha e^{-k}}.
\tag{UES$_{k>0}$}\label{eq:UES-kpos}
\]
\end{enumerate}
\end{proposition}

\begin{proof}
Recall that an analytic positive semigroup is uniformly exponentially stable if and only if it is weakly stable if and only if it is exponentially stable (use \cite[Thm.~V.1.10, Prop.~VI.1.14, and Ex.~V.1.6(4)]{EN00}). Since the operators $A$ and $D$ generate positive semigroups, as shown in the proof of Proposition~\ref{prop:9.6}, and since $s(D)<0$, $\ttt$ is uniformly exponentially stable if and only if the spectral bound of $(A+B\cd_0)$ is strictly negative (use Proposition~\ref{prop:8.10}(ii)).

Let $k=0$. By \eqref{eq:9.4} we have
\[
A+B\cd_0=\begin{pmatrix}
\alpha -1 & 1 \\
1 & \beta -1
\end{pmatrix},
\]
whose largest eigenvalue $\alpha+\beta-2+\sqrt{(\alpha-\beta)^2+4}$ is strictly negative if and only if \eqref{eq:UES-k0} holds.

We now consider the case $k\not=0$. Observe that for $\lambda=0$ we obtain $\mu_1=0$, $\mu_2=-k$, and hence \eqref{eq:9.3} yields
\[
B\cd_0=\frac{1}{1-e^{-k}}\begin{pmatrix}
-k & k \\
ke^{-k} & -ke^{-k}
\end{pmatrix}.
\]
The characteristic polynomial of $A+B\cd_0$ is then $\lambda^2 + b\lambda + c = 0$, where
\[
b = k(1+e^{-k})(1-e^{-k})^{-1}-\alpha-\beta \quad \text{and} \quad c= \alpha\beta - k(\alpha e^{-k} + \beta)(1-e^{-k})^{-1}.
\]
Note that
\[
b^2-4c = (\alpha-\beta -k)^2 + 4k^2e^{-k}(1-e^{-k})^{-2}> 0,
\]
and hence both eigenvalues of $A+B\cd_0$ are real. They are strictly negative if and only if the terms $b$ and $c$ are strictly positive, that is, if and only if
\[
(\alpha +\beta )(1-e^{-k})< k(1+e^{-k}) \qquad\text{and}\qquad \alpha\beta (1-e^{-k}) > k (\beta +\alpha e^{-k}). 
\]
\end{proof}

\begin{remark}\label{rem:9.13}
We conclude by briefly observing that the theory of osc operator matrices also fits second order problems. Let us consider the wave equation with dynamical boundary condition
\[
\begin{cases}
\ddot{u}(t,x)=u''(t,x), & t\geq0,\; x\in[0,1], \\
\ddot{u}(t,0)=u'(t,0)+\alpha u(t,0), & t\geq0, \\
\ddot{u}(t,1)=-u'(t,1)+\beta u(t,1), & t\geq0, \\
u(0,x)=f(x), & x\in[0,1], \\
\dot{u}(0,x)=g(x), & x\in[0,1].
\end{cases}
\tag{WDB}\label{eq:WDB}
\]
A system similar to \eqref{eq:WDB} has already been discussed in \cite{FGGR00,XL02} in the space $C[0,1]$, a setting that permits to interpret the dynamical boundary conditions defined above as generalized Wentzell (or Wentzell-Robin) boundary conditions. Thanks to our results in \eqref{sec:diffusion-transport} we can now prove that \eqref{eq:WDB} is well-posed.

Let us denote by $\delta'_x$ the linear form $\delta_x'f:=f'(x)$ on $H^2(0,1)$, $0\le x\le 1$. By Proposition~\ref{prop:9.7} and Proposition~\ref{prop:9.8} the operator matrix
\[
\begin{pmatrix}
0 & \begin{pmatrix} \delta_0' \\ -\delta_1' \end{pmatrix} \\
0 & \frac{d^2}{dx^2}
\end{pmatrix}
\quad\text{with domain}\quad\left\{\begin{pmatrix} \begin{pmatrix} x_0 \\ x_1 \end{pmatrix} \\ u \end{pmatrix}\in\mathbb{C}^2\times H^2(0,1) : u(0)=x_0,\; u(1)=x_1\right\}
\]
is dissipative and selfadjoint on the Hilbert space $\mathbb{C}^2\times L^2(0,1)$, and hence it generates a cosine function. It follows that there exists a unique classical solution of \eqref{eq:WDB} for all $f,g\in H^2(0,1)$, $\alpha,\beta\in\mathbb{C}$ (use \cite[Cor.~3.14.12 and Cor.~3.14.13]{ABHN01}).
\end{remark}

\begin{remark}\label{rem:9.14}
The matrix $\ca$ also defines an operator on the product space $\cx_p:=\mathbb{C}^2\times L^p(0,1)$, $1\le p<\infty$. Arguing as in the proof of Proposition~\ref{prop:9.5}, one can see that $\ca$ has compact resolvent on any space $\cx_p$ - in fact, this holds even after replacing the diffusion-transport operator $D$ by an arbitrary operator with compact resolvent on $L^p(0,1)$.

After the draft of this paper was completed, J. Goldstein informed us that this result has been obtained by Binding et al. (cf. \cite[\S~4]{BBW00}) only for $p=2$ in the context of Sturm–Liouville problems with eigenparameter dependent boundary conditions. We gratefully thank him for this observation.
\end{remark}

\end{document}